\documentclass[10pt]{article}
\font\smallit=cmti10

\usepackage[english]{babel}
\usepackage{amssymb,amsthm,amsmath,amsfonts}

\usepackage{bbold}
\usepackage{latexsym}
\usepackage{ifthen, bigstrut}
\usepackage{accents}
\usepackage[english]{babel}
\usepackage{hyperref}
\usepackage{breakurl}
\usepackage{ifthen, bigstrut}

\usepackage{enumerate}

\newtheorem{theorem}{Theorem}

\newtheorem{corollary}[theorem]{Corollary}
\newtheorem{lemma}[theorem]{Lemma}

\newcommand{\R}{\mathbb{R}}
\newcommand{\Np}{\mathbb{Np}}
\newcommand{\GS}{\mathbb{GS}}
\newcommand{\s}{\mathbb{S}}

\newcommand{\LS}{\mathit{Ls}}
\newcommand{\LSo}{\overline{\mathit{Ls}}}
\newcommand{\LSu}{\underline{\mathit{Ls}}}
\newcommand{\RS}{\mathit{Rs}}
\newcommand{\RSo}{\overline{\mathit{Rs}}}
\newcommand{\RSu}{\underline{\mathit{Rs}}}
\newcommand{\Ls}{\mathit{Ls}}
\newcommand{\Lso}{\overline{\mathit{Ls}}}
\newcommand{\Lsu}{\underline{\mathit{Ls}}}
\newcommand{\Rs}{\mathit{Rs}}
\newcommand{\Rso}{\overline{\mathit{Rs}}}
\newcommand{\Rsu}{\underline{\mathit{Rs}}}

\newcommand{\GL}{{G^\mathcal{L}}}
\newcommand{\GR}{{G^\mathcal{R}}}
\newcommand{\HL}{H^\mathcal{L}}
\newcommand{\HR}{H^\mathcal{R}}
\newcommand{\XL}{X^\mathcal{L}}
\newcommand{\XR}{X^\mathcal{R}}

\renewcommand{\ge}{\geqslant}
\renewcommand{\le}{\leqslant}
\newcommand{\Num}[1]{\widehat{#1}}
\newcommand{\pura}[2]{\langle\emptyset^{#1}\!\mid \!\emptyset^{#2}\rangle}
\newcommand{\atom}[1]{\emptyset^{#1}}
\newcommand{\Conj}[1]{\overset{\leftrightarrow}{#1}}
\newcommand{\ConjS}[1]{\overset{\longleftrightarrow}{#1}}
\newcommand{\GB}{\mathbb{GS}}

\newcommand{\su}{\succcurlyeq}
\newcommand{\pr}{\preccurlyeq}

\theoremstyle{definition}
\newtheorem{definition}[theorem]{Definition}

\newtheorem{observation}[theorem]{Observation}
\newtheorem{remark}[theorem]{Remark}

\begin{document}
\begin{center}
\uppercase{\bf Guaranteed Scoring Games}
\vskip 20pt
{\bf Urban Larsson\footnote{Supported by the Killam Trust}}\\
{\smallit Dalhousie University, Canada}\\
{\bf Jo\~ao P. Neto\footnote{Work supported by centre grant (to BioISI, Centre Reference: UID/MULTI/04046/2013), from FCT/MCTES/PIDDAC, Portugal.}}\\
{\smallit University of Lisboa, BioISIÐ Biosystems \& Integrative Sciences Institute}\\
{\bf Richard J.~Nowakowski\footnote{Partially supported by NSERC}}\\
{\smallit Dalhousie University, Canada}\\
{\bf Carlos P. Santos\footnote{Corresponding author: Center for Functional Analysis,
Linear Structures and Applications, Av. Rovisco Pais, 1049-001, Lisboa, Portugal; cmfsantos@fc.ul.pt}}\\
{\smallit Center for Functional Analysis, Linear Structures and Applications, Portugal}\\
\end{center}

\begin{abstract}
The class of Guaranteed Scoring Games (GS) are two-player combinatorial games 
with the property that Normal-play games (Conway et. al.) are ordered embedded into GS. 
They include, as subclasses,
 the scoring games considered by Milnor (1953), Ettinger (1996) and Johnson (2014). 
 We present the structure of GS and the techniques needed to analyze a sum of guaranteed games.
Firstly, GS  form a partially ordered monoid, via defined Right- and Left-stops over the reals, 
and with disjunctive sum as the operation. In fact, the structure is a quotient monoid 
with partially ordered congruence classes. We show that there are four reductions that when applied, 
 in any order, give a unique representative for each congruence class. The monoid is not a group, 
 but in this paper we prove that if a game has an inverse it is obtained by `switching the players'. 
 The order relation between two games is defined by comparing their stops in \textit{any} disjunctive sum. Here, we demonstrate how to compare the games via a finite algorithm instead, extending ideas of Ettinger, and also Siegel (2013).
 \end{abstract}

\section{Introduction} 
Combinatorial Game Theory (CGT) studies two-player games, (the players are called Left and Right)
 with perfect information
and no chance device. A common, almost defining feature, is that these games often 
decompose into sub-components and a player is only allowed to move in one of these at each stage of play. 
This situation is called a \textit{disjunctive sum} of games. 
It is also commonplace to allow addition of games with similar and well defined properties, games in such a family do not necessarily need to have the same rule sets.

The convention we wish to study, has the winner as the player with the best score. This convention includes rule sets such as \textsc{dots-\&-boxes}, \textsc{go} and \textsc{mancala}. A general, useful, theory has 
been elusive and, to our current knowledge, only four approaches appear in the literature.
Milnor \cite{Milno1953}, see also Hanner \cite{Hanne1959},
 considers dicot games (both players have a move from any non-terminal position) 
 with nonnegative incentive. In games with a nonnegative incentive, a move never worsens the player's 
 score; that is, zugzwang games, where neither player wishes to move, do not appear. 
Ettinger \cite{Ettin1996,Ettin2000} considers all dicot games. Stewart \cite{Stewa2011} defines a comprehensive class but
it has few useful algebraic properties. Johnson \cite{Johns2014} considers another subclass of dicot games, 
for which, for any position, the lengths of every branch of the game tree has the same parity. 

We study the class of \textit{Guaranteed Scoring Games}, $\GS$ which were 
introduced in \cite{LarssNS}. This class has a partial order relation, $\su$, 
which together with the disjunctive sum operation induces a congruence relation $(\sim, +)$. 
The resulting quotient monoid inherits partially ordered congruence classes, and it is the purpose of this paper to continue the 
study of these classes. 
In \cite{LarssNS}, it was shown that Normal-play games (see Remark~\ref{rem1}) can be ordered embedded in a natural way 
and that a positive incentive for games without Right or Left options 
is an obstacle to the order embedding. It was also demonstrated how to compare games with numbers
using waiting moves (images of Normal-play integers) and pass-allowed stops. 
Intuitively, this class of games has the property that the 
players want the component games to continue; every game in which at least one player cannot 
move has non-positive incentive.

Here we show that $\GS$ has the properties:

\begin{enumerate}
\item  There is a constructive way to give the order relation between games $G$ and $H$. 
It only requires $G$, $H$ and a special type of simplistic games that we call `waiting moves', games with the 
sole purpose of giving one of the player an extra number of moves, but with no change in score. 
\item There are four reduction theorems, and we find a unique representative game for each congruence class.
Of these, `Eliminating Dominated Options' and 
  `Bypassing Reversible Moves'
  with a non-empty set of options are analogous to those found in the theory of Normal-play
 games. Bypassing a reversible move by just replacing it with an empty-set of options leads to a non-equivalent game. 
 In this case, the appropriate reduction requires consideration of the pass-allowed stops. 
 This has no corresponding 
  concept in Normal-play theory. 
\item 
In $\GS$ the Conjugate Property holds: if $G+H$ is equivalent 
 to 0 then $H$ is the game obtained by interchanging the roles of Left and Right. In Normal-play, this is 
 called the negative of $G$; however in $\GS$ `negatives' do not always exist.
\item We solve each of these problem via a finite algorithm, which is also implemented in a Scoring Games Calculator.
\end{enumerate}

The organization of the paper is as follows: 
Section~\ref{sec:background} introduces the main concepts for Combinatorial Games. In
Section~\ref{sec:guaranteed}, 
 the class of Guaranteed Scoring Games,  together with the order relations and congruence classes, is presented.
 Section~\ref{sec:normal} contains 
 results concerning the order embedding of Normal-play games in $\GS$. Section~\ref{sec:scores}
presents results on pass-allowed stops and waiting moves.   Section~\ref{sec:3reductions} proves
  four reductions that simplify games. Section~\ref{sec:uniqueness} proves that applying 
 these reductions leads to a unique game. The proofs require extending
 Siegel's `linked' concept for mis\`ere games to scoring games which is in Section \ref{sec:gamecomparison}.
 Section~\ref{sec:conjugates} shows that the Conjugate Property holds in $\GS$. In Section~\ref{sec:calc} we give a brief intro to the Scoring Games Calculator.

\begin{remark}\label{rem1} 
Other famous winning conditions in CGT are considering who moves 
last. \textit{Normal-play} games, the first player who cannot move loses, 
find their origins with the analysis of \textsc{nim} \cite{bouto1902}; see also \cite{Grund1939,Sprag}.
Conway developed the first encompassing theory; see
\cite{ BerleCG2001--2004,Conwa2001}. A comprehensive \textit{Mis\`ere} theory,
the first player who cannot move wins,
has not yet been developed but large strides have been made for impartial games,
see \cite{PlambS2008}. A related winning convention arises in the Maker-Breaker (or Maker-Maker) games usually 
played on a graph---one player wishes to create a structure and the opponent wants to stop this (or both want to create a structure)
such as \textsc{hex} or generalized \textsc{tic-tac-toe}. See Beck \cite{Beck2006} for more details. 
\end{remark}

%%%%%%%%%%%%%%%%%%%%%%%%%
%%%%%%%%%%%%%%%%%%%%%%%%%
\section{Background}\label{sec:background}

%%%%%%%%%%%%%%%%%%%%%%%%%
For any combinatorial game $G$ (regardless of the winning condition) there are two players
who, by convention, are called \textit{Left} (female) and \textit{Right} (male)\footnote{Remember, \textit{L}ouise and \textit{R}ichard Guy who have contributed much to combinatorial games.}.
From $G$, a position that some player can move to (in a single move)
 is an \textit{option} of $G$.
The \textit{left} options are those to which Left can move and the corresponding set
is denoted by $\GL$. An element of $\GL$ is often denoted by $G^L$. Similarly, there is a
set of \textit{right} options denoted by $\GR$, with a typical game $G^R$. There is no requirement that $\GL$ and
$\GR$ be disjoint. A game can be recursively defined in terms of its options. 
We will use the representation $G=\langle \GL\mid \GR\rangle$ (so as to distinguish them 
from Normal-play games where the convention is $\{\GL\mid\GR\}$).
The \textit{followers} of $G$ are defined recursively: $G$ and all its options are followers of $G$ and each follower of a follower of $G$ is a follower of $G$.
The set of \textit{proper} followers of $G$ are the followers except for $G$ itself. 
The \textit{game tree} of a position $G$ would then consist of all the followers of $G$
 drawn recursively: i.e. the options of a follower $H$ of $G$ are the children of $H$ in the tree.

Positions fall into two kinds:  \textit{atomic} positions in which at most one player can move, 
and \emph{non-atomic} positions in which both players can move. A position with no Left options
is called \textit{left-atomic}, and in case of no Right options it is \textit{right-atomic}. 
A game  with no options at all is called \textit{purely-atomic}, that is, such games are both left-atomic and right-atomic.\\

%%%%%%%%%%%%%%%%%%%%%%%%%
\subsection{Introduction to Guaranteed Scoring Games}\label{sec:guaranteed}

 In scoring combinatorial games, the score of a game is determined at the end of the game, that is when the player to move has no option.
 
\begin{definition}[Game termination]\label{def:ending}
Let $G$ be a left-atomic game. We write $\GL=\emptyset^{\ell}$, $\ell\in \R$
to indicate that, if Left to move, the game is over and the \emph{score} is the real number $\ell$.
Similarly, if $G$ is right-atomic then $\GR=\emptyset^r$, and if
 it is Right's move then there are no Right options and
the \emph{score} is $r\in \R$. Left wins if the score is positive, 
Right wins if the score is negative, and 
it is a tie if the score is zero.
\end{definition}

Since the game $\langle \emptyset^s\!\mid \!\emptyset^{s} \rangle$ results in a score of $s$ regardless of whose turn it is, we call this game (the number) $s$.
We refer to the adorned empty set, $\emptyset^s$, $s\in \R$, as an \textit{atom} or, if needed for 
specificity, the $s$\textit{-atom}. By  an \emph{atom in a game} $G$, we mean an atom of some 
atomic
follower of $G$. By an atom in a set of games we mean an atom in one of the games in that set. In the general scoring universe,  denoted by $\s$ (see also \cite{LarssNS,Stewa2011}), there is no restriction to the form of the atomic games.

\begin{definition}\label{def:guaranteed}
A game $H\in \s$ is \textit{guaranteed} if, for every atomic follower $G$, 
\begin{itemize}
\item[1.] if $G=\langle \atom{\ell} \mid \emptyset^r \rangle $ then $\ell \le r$;
\item[2.] if $G=\langle \emptyset^\ell \mid \GR \rangle $ then $\ell \le s$, for every $s$-atom in $G$;
\item[3.] if $G=\langle \GL \mid \emptyset^r \rangle $ then $s \le r$, for every $s$-atom in $G$.
\end{itemize}
\end{definition}

Note that this definition is formally equivalent to: a game $H\in \s$ is \textit{guaranteed} if, for every atomic follower $G$, $\ell \le r$, for all $\ell$-atoms in $\GL$ and all $r$-atoms in $\GR$. Moreover, observe that if $H$ is guaranteed then every follower of $H$ is also guaranteed. For example, 
$\langle\langle\emptyset^1\mid \emptyset^1 \rangle\mid \langle\emptyset^0\mid \emptyset^0 \rangle \rangle = \langle1\mid 0 \rangle$ is guaranteed but 
$H=\langle \emptyset^1 \mid\mid 4, \langle \emptyset^3 | 3, \langle \emptyset^5\mid
4\rangle\rangle$ is not since:

(a) $H$ is left-atomic and 1 is less than the other  scores;

(b) both $3=\langle\emptyset^3\mid\emptyset^3\rangle$ and $4= \langle\emptyset^4\mid\emptyset^4\rangle$ are atomic and each satisfies (i) in Definition~\ref{def:guaranteed};

(c) $\langle \emptyset^3 \mid 3, \langle \emptyset^5\mid4\rangle\rangle$ is left-atomic and
3 is less than or equal to 3,  5 and 4;

(d) $\langle \emptyset^5\mid 4\rangle  = \langle \emptyset^5\mid \langle\emptyset^4
\mid\emptyset^4\rangle\rangle$ does not satisfy (ii) and thus $H$ is not guaranteed.\\

The class of Guaranteed Scoring Games, $\GS$, can be defined directly as the class that contains all
guaranteed games. We give an equivalent recursive definition.
\begin{definition}\label{def:recgua}   
Let $\GS_0$ be the set of birthday 0 guaranteed games. These are of the form 
$\{\langle \, \emptyset^{\ell}\mid \emptyset^r \, \rangle : \ell,r\in \mathbb{R}, \ell\leq r \}$. Suppose that $\cal G$ and $\cal H$ are sets of 
guaranteed games of birthday less than $i$. 
The set of non-atomic games of the form $\langle \,  {\cal G}\mid {\cal H} \, \rangle$ together with atomic games of the forms 
$\pura{\ell}{r}$, $\langle \emptyset^\ell\mid {\cal H} \rangle$ and 
$\langle {\cal G}\mid \emptyset^r\rangle$ are the games in $\GS_{i}$. For $i>0$, if $G\in \GS_{i}\setminus \GS_{i-1}$ 
then $G$ is said to have \textit{birthday} $i$ and we write $b(G)=i$. 
\end{definition}
It follows that $\GS=\cup_{i\ge 0}\GS_i$, with notation as in Definition~\ref{def:recgua}. 
The birthday of a game corresponds to the depth of its game tree. This stratification into birthdays  is 
very useful for proofs by induction.

A player may be faced with several component games/positions, 
and if there is at least one in which he can move then he has an option and the game is 
not over yet. 
A move in a disjunctive sum of positions is a move in exactly one of the component positions, 
and the other ones remain unchanged. It is then the other playerÕs turn to move. 
We formalize this in the next definition by listing all the possible cases. The distinction between the two uses of $+$, the disjunctive sum of games and the addition of real numbers, will always be clear from the context.
 If $\mathcal{G}=\{G_1, \ldots,G_m\}$ is a set of games and $H$ is a single game then 
 $\mathcal{G}+H =\{G_1+H, \ldots, G_m+H\}$ if $\mathcal{G}$ is non-empty; 
 otherwise $\mathcal{G}+H$ is not defined and will be removed from any list of games.

  An intuitively obvious fact that is worthwhile highlighting at this point:
if Left has no move
in $G+H$ then Left has no move in neither of $G$ and  $H$ (and reverse), that is:
$$\textit{ $G+H$ is left-atomic if and only if both $G$ and $H$ are left-atomic,}$$
and analogously for right-atomic games.

\begin{definition}\label{def:disjunctive}[Disjunctive Sum]
The disjunctive sum of two guaranteed scoring games $G$ and $H$ is given by:
\[G+ H=  \begin{cases} 
\langle \, \emptyset^{\ell_1+\ell_2}\mid\emptyset^{r_1+r_2} \, \rangle, \quad\textrm{ if $G=\langle \, \emptyset^{\ell_1}\mid\emptyset^{r_1} \, \rangle$ and
$H=\langle \, \emptyset^{\ell_2}\mid\emptyset^{r_2} \, \rangle$;}\\
\langle \, \emptyset^{\ell_1+\ell_2}\mid\GR + H,G+\HR \, \rangle, \textrm{ if
$G=\langle \, \emptyset^{\ell_1}\mid\GR  \, \rangle$  and
$H=\langle \, \emptyset^{\ell_2}\mid\HR \, \rangle$},\\
{}\qquad \textrm{ and at least one of $\GR $ and $\HR$
is non-empty;}\\
\langle \, \GL + H,G+\HL\mid \emptyset^{r_1+r_2} \, \rangle, \textrm{ if
$G=\langle \, \GL \mid\emptyset^{r_1} \, \rangle$  and
$H=\langle \, \GL \mid\emptyset^{r_2} \, \rangle$},\\
{}\qquad \textrm{ and at least one of $\GL $ and $\HL$
is non-empty;}\\
\langle \, \GL + H,G+ \HL\mid\GR + H,G+\HR \, \rangle,
\textrm{ otherwise.}
\end{cases}
\]
\end{definition}
Note that in the last equality, if there are no left options in $G$, then $\GL + H$ gets removed, unless both $\GL$ and $\HL$ are atoms, in which 
case some earlier item applies.

%%%%%%%%%%%%%%%%%%%%%%%%%
%%%%%%%%%%%%%%%%%%%%%%%%%
\begin{theorem}\label{thm:monoidstructure}
$(\GB,+)$ is a commutative monoid.
\end{theorem}

\begin{proof} In all cases, the proof is by induction on the sum of the birthdays of the positions.
\begin{enumerate}
  \item $\GB$ is closed, that is, $G,H\in\GB\Rightarrow G+H\in\GB$.

Suppose that $G+H$ is left-atomic. Then both $G=\langle \emptyset^{g} \mid \GR\rangle$ and $H=\langle \emptyset^{h} \mid \HR\rangle$ are left-atomic. 
Since both games are guaranteed, then each $s$-atom in $G$ satisfies $g\le s$ and each $t$-atom in $H$ satisfies $h\le t$. Therefore $g+h\le \min \{s+t\}$, 
and so $G+H=\langle \emptyset^{g+h} \mid (G+H)^{\mathcal{R}}\rangle$ is also guaranteed. this case includes the possibility that $(G+H)^{\mathcal{R}}$ 
is the $(s+t)$-atom. Finally, suppose that both $G^{\mathcal{L}}$ and $G^\mathcal{R}$ are 
non-empty sets of games of $\GS$. Both players have moves in $G+H$ that, by 
induction, are games of $\GS$. So, $G+H\in \GS$.

  \item Disjunctive sum is commutative.

 If $G=\langle\emptyset^{\ell_1}\mid\emptyset^{r_1}\rangle$ and
$H=\langle\emptyset^{\ell_2}\mid\emptyset^{r_2}\rangle$ then $G+H=\langle \emptyset^{\ell_1+\ell_2}\mid\emptyset^{r_1+r_2}\rangle=\langle \emptyset^{\ell_2+\ell_1}\mid\emptyset^{r_2+r_1}\rangle=H+G$.

If $G=\langle\emptyset^{\ell_1}\mid\GR\rangle$ and $H=\langle\emptyset^{\ell_2}\mid\HR\rangle$ then
\[G+H=\langle\emptyset^{\ell_1+\ell_2}\mid\GR +H,G+\HR\rangle\underbrace{=}_{induction}\langle\emptyset^{\ell_2+\ell_1}\mid H+\GR,\HR+G\rangle=H+G.\]

The other cases are analogous using induction and the fact that the addition of real numbers is commutative.
  \item Disjunctive sum is associative.

 If $G=\langle\emptyset^{\ell_1}\mid\emptyset^{r_1}\rangle$,
$H=\langle\emptyset^{\ell_2}\mid\emptyset^{r_2}\rangle$ and $J=\langle\emptyset^{\ell_3}\mid\emptyset^{r_3}\rangle$ then $G+(H+J)=(G+H)+J$ is just a consequence of that the addition of real numbers is associative.

If $G=\langle\emptyset^{\ell_1}\mid\emptyset^{r_1}\rangle$,
$H=\langle\emptyset^{\ell_2}\mid\emptyset^{r_2}\rangle$ and $J=\langle\emptyset^{\ell_3}\mid J^{\mathcal{R}}\rangle$ then
\begin{eqnarray*}
G+(H+J)&=&\langle\emptyset^{\ell_1}\mid\emptyset^{r_1}\rangle+\langle\emptyset^{\ell_2+\ell_3}\mid H+J^{\mathcal{R}}\rangle\\
&=&\langle\emptyset^{\ell_1+(\ell_2+\ell_3)}\mid G+(H+J^{\mathcal{R}})\rangle\\
&\underbrace{=}_{induction}&\langle\emptyset^{(\ell_1+\ell_2)+\ell_3}\mid (G+H)+J^{\mathcal{R}}\rangle=(G+H)+J.
\end{eqnarray*}

The other cases are  analogous using induction and the fact that the addition of real numbers is associative.

    \item It follows directly from the definition of disjunctive sum that $G+0=0+G=G$
    so the identity of $(\GB,+)$ is $0$.
\end{enumerate}
\end{proof}
When analyzing games, the following observation, which follows 
from the definition of the disjunctive sum, is useful for human players.
\begin{observation}[Number Translation] 
Let $G\in\GS$ and $x\in\mathbb{R}$ then 
\[G+x = \begin{cases} \langle \emptyset^{\ell+x}\mid\emptyset^{r+x}  \rangle
\mbox{ if $G = \langle \emptyset^{\ell}\mid\emptyset^{r}\rangle$,}\\
\langle \emptyset^{\ell+x}\mid\GR+x  \rangle
\mbox{ if $G = \langle \emptyset^{\ell}\mid\GR \rangle$,}\\
\langle \GL+x\mid\emptyset^{r+x}  \rangle
\mbox{ if $G = \langle \GL\mid\emptyset^{r}  \rangle$,}\\
\langle \GL+x\mid\GR+x  \rangle
\mbox{ if $G = \langle \GL\mid\GR  \rangle$.}
\end{cases}	
\]
\end{observation}

Next, we give the fundamental definitions for comparing games. 
\begin{definition} \label{def:stops}
For a game $G\in \GS$:
\[\LS(G) = \begin{cases}
\ell &\mbox{if } G=\langle \emptyset^\ell\mid  \GR\rangle  \\
\max\{\RS(G^L) : G^L\in \GL\} & \mbox{otherwise;} \end{cases}
\]
and
\[\RS(G) = \begin{cases}
r&\mbox{if } G=\langle \GL\mid  \emptyset^r  \rangle  \\
\min\{\LS(G^R) : G^R\in \GR\} & \mbox{otherwise.} \end{cases}
\]
\end{definition}

We call $\LS(G)$ the \emph{Left-stop} of $G$ and $\RS(G)$ the \emph{Right-stop} of $G$.
%RJN commented out the footnote.
%\footnote{Note that the term `stop' is commonly used in Normal-play games, but it has a different meaning here. If each 
%game component with non-empty sets of options score zero, then we stop playing and sum up the scores of the already 
%terminated game components.}

    \begin{definition}\label{def:equivalence}(Inequalities for games)\\
Let $G, H\in \GB$. Then
    $G\succcurlyeq  H$ if for all $ X\in \GB$ we have
    $\LS(G+X)\ge  \LS(H+X)\, \mbox{ and }\,
    \RS(G+X)\ge  \RS(H+X)$.
The games $G$ and $H$ are \textit{equivalent}, denoted by $G\sim H$,
if $G\succcurlyeq  H$ and $H\succcurlyeq  G$.
    \end{definition}
%RJN commented out the next paragraph.
%Note that the ordered pairs $(\Ls(G),\Rs(G))\in \mathbb R\times \mathbb R$ may be regarded as the partially ordered 
%outcomes of games $G\in \GS$. In particular $G\sim H$ means that $(\Ls(G+X),\Rs(G+X)) = (\Ls(H+X),\Rs(H+X))$, for 
%all $X\in \GS$. %which is quite a strong requirement, given the size of $\mathbb R$.

\begin{theorem}\label{thm:orderstructure}
The relation $\succcurlyeq$ is a partial order and $\sim$ is an equivalence relation.
\end{theorem}

\begin{proof}
\noindent
Both assertions follow directly from their definitions and the fact that  the reals are totally ordered.
\end{proof}

Theorem~\ref{thm:orderstructure} shows that the monoid $(\GS,+)$ can be regarded as the 
algebraic structure $(\GS,+,\su)$. The next three results show that  $(\GS,+)$ modulo $\sim$ is a quotient monoid, and that in fact $(\sim, +)$ is a congruence relation; the additive structure on the equivalence classes $(\GS,+)\!$ modulo $\!\sim$ is inherited from $(\GS,+)$. (A natural function from the congruence classes to the outcomes can be obtained via the unique representatives which define the canonical forms  as discussed in Section~\ref{sec:uniqueness}.)

\begin{lemma}\label{lem:oneadded}
 Let $G, H\in \GS$. If $G\succcurlyeq H$ then
$G+J\succcurlyeq H+J$ for any $J\in \GS$.
\end{lemma}

\begin{proof}
Consider any game $J\in \GS$. Since $G\succcurlyeq H$, it follows that, $\LS(G+(J+X))\ge \LS(H+(J+X))$, for any $X\in \GS$.
Since disjunctive sum is associative this inequality is the same as
$\LS((G+J)+X))\ge \LS((H+J)+X)$. The same argument gives
$\RS((G+J)+X))\ge \RS((H+J)+X)$ and thus, since $X$ is arbitrary, this gives that $G+J\su H+J$.
\end{proof}

\begin{corollary}\label{cor:disjunctiveorder}
Let $G,H,J,W\in \GS$. If
$G\succcurlyeq H$ and $J\succcurlyeq W$ then $G+J\succcurlyeq H+W$.
\end{corollary}
\begin{proof}
Apply Lemma~\ref{lem:oneadded} twice.
\end{proof}

\begin{corollary}\label{cor:congruence}
Let $G,H,J,W\in \GS$. If
$G\sim H$ and $J\sim W$ then $G+J\sim H+W$.
\end{corollary}
\begin{proof}
Since $X\sim Y$ means $X\su Y $ and $Y\su X$, the result follows by applying Corollary~\ref{cor:disjunctiveorder} twice.
\end{proof}

The \textit{conjugate}  of a game $G$, $\Conj{G}$, is defined recursively:
$\Conj{G}=\langle \ConjS{\GR} \! \mid \! \ConjS{\GL}\rangle$, where $\ConjS{\GR}$ means
$\ConjS{G^R}$, for each $G^R\in \GR$ (and similarly for $\GL$),
unless $\GR=\emptyset^r$, in which case $\ConjS{\GR}=\emptyset^{-r}$. It is easy to see that if a game is guaranteed, then its conjugate is also. 
As mentioned early, this is equivalent to interchanging Left and Right. In Normal-play
$G+\Conj{G} \sim 0$, but in $\GS$ this is not necessarily true. For example, if $G=\langle \emptyset^\ell \mid \emptyset^r \rangle$,
  then conjugate is $\Conj{G} = \langle \emptyset^{-r}  \mid \emptyset^{-\ell} \rangle$
  and $G+\Conj{G}\sim 0$ if  and only if $\ell=r$. 

The next two results will be useful in proving the Conjugate Property in Section~\ref{sec:conjugates}.
\begin{lemma}\label{lem:strict}
Let $G,H\in{\GS}$. If $G\succ 0$ and $H\succcurlyeq 0$ then $G+H\succ 0$.
\end{lemma}
\begin{proof}
By Corollary~\ref{cor:disjunctiveorder}, we already know that $G+H\succcurlyeq 0$. So, it is enough to show that $G+H\nsim 0$.
Since $G\succ 0$ then, without loss of generality, we may assume that
 $\LS(G+X)>\LS(X)$ for some $X$. Because $H\succcurlyeq 0$,
 we have $\LS(G+X+H)\ge \LS(G+X+0)=\LS(G+X)>\LS(X)$, and
 therefore $G+H\nsim 0$.
\end{proof}

\begin{lemma}\label{lem:dis}
Let $G,H\in{\GS}$. Let $J\in \GS$ be invertible, then
$G+J\su H+J$ if and only if
 $G\su H$.
\end{lemma}

\begin{proof}
 The direction $\Rightarrow$ follows immediately from Lemma~\ref{lem:oneadded}.

Consider $J,X\in \GS$ where $J$ is invertible.
Consider $X'=X+J'$ where $J'$ is the inverse of $J$. Because $G+J\su H+J$,
we have $\LS(G+X)=\LS(G+J+X')\ge \LS(H+J+X')=\LS(H+X)$ and
$\RS(G+X)=\RS(G+J+X')\ge \RS(H+J+X')=\RS(H+X)$.
\end{proof}

%%%%%%%%%%%%%%%%%%%%%%%%%
%%%%%%%%%%%%%%%%%%%%%%%%%
\subsection{Relation between Normal-play and Guaranteed Games}\label{sec:normal}

%%%%%%%%%%%%%%%%%%%%%%%%%
%%%%%%%%%%%%%%%%%%%%%%%%%

One of the main results in \cite{LarssNS} is that Normal-play games are order-embedded in $\GS$.

\begin{definition}\label{defn:naturalembed}
For a Normal-play game $G$, let $\Num{G}$ be the scoring game obtained by replacing
each empty set, $\emptyset$, in $G$ by the atom $\emptyset^0$.
\end{definition}

This operation retains the game tree structure. For example, the leaves of a Normal-play game tree are labelled
$0=\{\emptyset\mid \emptyset\}$ which is replaced by
$0=\langle\emptyset^0\mid\emptyset^0\rangle$ for the scoring game.

\begin{theorem}[\cite{LarssNS}]\label{thm:order}
Let $\Np$ be the set of Normal-play games. The set $\{\Num{G}:G\in \Np\}$
induces an order-embedding of $\Np$ in $\GS$.
\end{theorem}

That is, $G\su H$ in Normal-play if and only if $\Num{G}\su \Num{H}$ in guaranteed games.

Let $n$ be an integer. The games $\Num{n}$ are called \textit{waiting moves}.
For example, $\Num{0}=\langle\emptyset^0\mid\emptyset^0\rangle = 0$ and
$\Num{1}=\langle 0\mid \emptyset^0\rangle$ and
$\Num{2}=\langle \Num{1}\mid \emptyset^0\rangle$. Regardless, the score of a waiting move
will be 0, but in a game $G+\Num{1}$, Left has the ability to force Right to play consecutive moves in the
$G$ component.

  The ability to pass may appear as something beneficial for 
a player. This is true in $\GS$ but not necessarily in the general universe of scoring games. 
 For example, let
 $G=\langle\emptyset^1\mid \langle\emptyset^{-9}\mid\emptyset^9\rangle\rangle$
  and note $G\not\in \GS$. Clearly Left wins playing first. 
 In $G+\Num{1}$,
 Left has no move in $G$ and she must play her waiting move, $\Num{1}$.
 Right then plays to $\langle\emptyset^{-9}\mid\emptyset^9\rangle$.
 Now Left has no move and the score is  $-9$, a Right win. 

There are useful inequalities relating Normal-play and
Scoring games.

\begin{definition}\label{def:projection}
Let $G\in \GS$, and let $G_x$ be as $G$, but with each
 atom replaced by $\emptyset^x$. Let
  $max(G) = \max \{ s\,|\, \emptyset^s \text{ is an atom in } G\}$
and $min(G)=\min \{s\,|\,  \emptyset^s \text{ is an atom in } G\}$.
Set $G_{min}=G_{min(G)}$ and $G_{max}=G_{max(G)}$.
\end{definition}

\begin{theorem}[Projection Theorem]\label{thm:projection}
\ Let $G\in \GS$. If $n =b(G)$ then
\begin{enumerate}
\item $G_{min}\pr G\pr G_{max}$
\item $min(G) -\widehat n \pr G\pr max(G) +\widehat n$
\item $b\left(min(G) -\widehat n\right)=b\left(max(G)+\widehat n\right)=n$.
\end{enumerate}
\end{theorem}
\begin{proof}
For part~1, for any $X$, we establish the inequalities $\LS(G_{min}+X)\le \LS(G+X)$ and $\RS(G_{min}+X)\le \RS(G+X)$. First, if the game $G+X$ is purely atomic, then, so is $G_{min}+X$, and the inequalities are trivial, given Definition~\ref{def:projection}.

Consider the game, $(G_{min}+X)^L$, obtained after an optimal move by Left. Ignoring the scores, Left can make exactly the same move in the game $G+X$, to say $(G+X)^L$. Because, we maintain an identical tree structure of the respective games, we get
$$\Ls\left(G_{min}+X\right)=\Rs\left((G_{min}+X)^L\right)\le \Rs\left((G+X)^{L}\right)\le \Ls\left(G+X\right),$$
by induction.

To prove the inequality for the Right scores, we consider the game $(G+X)^R$, obtained after an optimal move by Right. Ignoring the scores, Right can make exactly the same move in the game $G_{min}+X$, to say $(G_{min}+X)^R$. Therefore $$\Rs\left(G_{min}+X\right)\le \Ls\left((G_{min}+X)^R\right)\le \Ls\left((G+X)^{R}\right)= \Rs\left(G+X\right),$$
by induction.

For part~2, it suffices to prove that $\min(G)-\Num{N} \pr G$ (and the proof of second inequality 
is similar). It is easy to see that $\Num{N}-\Num{N} \sim 0$. Therefore, it suffices to prove that 
$\min(G) \pr G+\Num{N}$, which holds if and only if $\min(G)\le \Lsu(G+\Num{N})$ and the latter
 is easy to see. Part~3 follows by definition of waiting-moves.
\end{proof}

%%%%%%%%%%%%%%%%%%%%%%%%%%%%
\subsection{Pass-allowed stops and Waiting moves}\label{sec:scores}
%%%%%%%%%%%%%%%%%%%%%%%%%%%%

The following three points about the stops are immediate from the definitions but
we state them explicitly since they will appear in many proofs.

\begin{observation}\label{obs:sets}
Given a game $G\in \GS$,

\noindent (i) $\LS(G) \ge \RS(G^L)$ for all $G^L$, and there is some
$G^L$ for which $\LS(G) = \RS(G^L)$;

\noindent (ii) $\RS(G) \le \LS(G^R)$ for all $G^R$, and there is
some $G^R$ for which $\RS(G) = \LS(G^R)$;

\noindent (iii) $\LS(G+s) = \LS(G)+s$ for any number $s$.
\end{observation}

The next result indicates that we only need to consider one of $\Ls$ and $\Rs$ for 
game comparison in $\GS$. However, in the sequel, the proofs that use induction on the birthdays need the inequalities for both the Left- and Right-stops, 
because we must consider games with a \emph{fixed} birthday. However, 
Theorem~\ref{thm:LRscores} enables a simple proof of Lemma~\ref{dl2}.

\begin{theorem}\label{thm:LRscores}
Let $G, H\in \GS$. Then $\LS(G+X)\ge \LS(H+X)$ for all
 $X\in \GS$ if and only if $\RS(G+Y)\ge \RS(H+Y)$ for all $Y\in \GS$.
\end{theorem}

\begin{proof}
The proof depends on the following result.

\textit{Claim~1:} Given $G$ and $H$ in $\GS$, then there exists $X\in \GS$
 such that $\LS(G+X)> \LS(H+X)$ iff there exists a game $Y\in \GS$ such that
$\RS(G+Y)> \RS(H+Y)$.\\

\noindent \textit{Proof of Claim~1.}
Suppose that there is some $X$ such that $\LS(G+X)>\LS(H+X)$.
Let $M=\max\{\LS(G+X) -\RS(G^R) : G^R\in \GR\}$ and let $G^{R'}$  be an option where
$M=\LS(G+X) -\RS(G^{R'}) $.
Put
$Y=\langle M \mid   X\rangle$.

Now, $\LS(G^R+Y)\ge \RS(G^R+M)$, since $M$ is a Left option of $Y$.
For any $G^R\in \GR$,
\begin{eqnarray*}
\RS(G^R+M)&=&\RS(G^R)+M,\mbox{ by Observation~\ref{obs:sets} (iii),} \\
&=& \RS(G^R)+\LS(G+X) -\RS(G^{R'})\\
&=& \LS(G+X)+\RS(G^R)-\RS(G^{R'})\\
&\ge& \LS(G+X).
\end{eqnarray*}
Therefore,  $\LS(G^R+Y)\ge \LS(G+X)$.
 Thus
\[\RS(G+Y) = \min\{\LS(G+X), \LS(G^R+Y)   : G^R\in \GR\} =  \LS(G+X)\]

Now,
$\RS(G+Y) =  \LS(G+X)>\LS(H+X)\ge \RS(H+Y)$
where the first inequality follows from the assumption about $X$, and, since
$X$ is a Right option of $Y$,
the second inequality follows from Observation~\ref{obs:sets} (ii).

 \textit{End of the proof of Claim~1.}\\

Suppose
$\LS(G+X)\ge \LS(H+X)$ for all $X$. By Claim~1, there is no game $Y$ for which
 $\RS(G+Y)<\RS(H+Y)$, in other words, $\RS(G+Y)\ge \RS(H+Y)$ for all $Y$.
 \end{proof}

In the next definition,
``pass-allowed'' typically means that one player has an arbitrary number of waiting moves in another component.

    \begin{definition}[\cite{LarssNS}]\label{def:strongstop}
    Let $G\in \GS$. Then
    $\LSu(G)=\min\{\LS(G-\Num{n}\, ):n\in\mathbb{N}_0\}$ is
    \emph{Right's pass-allowed Left-stop} of $G$.
    \emph{Left's pass-allowed Right-stop} is defined analogously,
    $\RSo(G)=\max\{\RS(G+\Num{n}):n\in\mathbb{N}_0\}$. We also define
    $\LSo(G)=\max\{\LS(G+\Num{n}\, ):n\in\mathbb{N}_0\}$ and
    $\RSu(G)=\min\{\RS(G+\Num{n}\, ):n\in\mathbb{N}_0\}$.
    \end{definition}
The `overline' indicates that Left can pass and the `underline' that Right can pass. 
Note that, in $\LSo(G)$, Left can even start by passing.

\begin{lemma}\label{lem:waiting moves}
Let $G\in \GS$. 
If $n\ge b(G)$ then $\LSu(G)=\LS(G-\Num{n})$ and $\RSo(G)=\RS(G+\Num{n})$.
\end{lemma}
\begin{proof}
Suppose that $n\ge b(G)$. By Theorem~\ref{thm:order}, we have $G - \Num{b(G)} \su G - \Num{n}$ which gives $\LS(G - \Num{b(G)}) \ge \LS (G - \Num{n})\ge \min \{\LS (G - \Num{m})\}=\LSu(G)$. Since Left begins, Right does not require more than $b(G)$ waiting-moves, until Left has run out of moves in $G$. Hence $\LSu(G) = \min \{\LS (G - \Num{m})\} \ge \LS(G - \Num{b(G)})$. This proves the first claim, and the claim for the Right-stop is analogous.
\end{proof}

From this result, it follows that the first part of Definition~\ref{def:strongstop} is equivalent to: for $G\in \GS$ and $n= b(G)$, $\LSu(G)=\LS(G-\Num{n})$ and $\RSo(G)=\RS(G+\Num{n})$. 
The pass-allowed Left- and Right-stops of a disjunctive sum of games can be bounded by pass-allowed stops of the respective game components.

\begin{theorem}\label{thm:pallowscores} (Pass-allowed Stops of Disjunctive Sums)\\
For all  $G,H\in{\GS}$ we have $$\LSu(G)+\RSu(H)\le\LSu(G+H)\le \LSu(G)+\LSu(H).$$
Symmetrically $$\RSo(G)+\RSo(H)\le\RSo(G+H)\le \RSo(G)+\LSo(H).$$
\end{theorem}

\begin{proof}
\noindent Let $n=b(G)$ and $m=b(H)$ and let $N=n+m$.

 Right plays second in $G+H-\widehat{N}$ which he can
regard as $(G-\widehat{n}) + (H-\widehat{m})$. He can restrict his moves
to responding only in the component in which Left has
just played and he has enough waiting moves to force Left to start both components.
Thus he can achieve
$\LSu(G)+\LSu(H)$, i.e., $\LSu(G+H)\le\LSu(G)+\LSu(H)$.

In the global game $G+H$, suppose that Right responds in $H$ to Left's first move in $G$,
then,  for the rest of the game, Left can copy each local move in the global setting and has enough
waiting moves to achieve a score of $\LSu(G)+\RSu(H)$. Since she has other strategies,
we have $\LSu(G)+\RSu(H)\le\LSu(G+H)$. The other inequality is proved analogously.
\end{proof}

The results for the rest of the paper are sometimes stated only for Left. 
The proofs for Right are the same with the roles of Left and Right interchanged.

\begin{corollary}\label{cor:emptywaiting}
Let $G,H\in{\GS}$. If
$H=\left\langle \emptyset^h\! \mid \! \HR\right\rangle$
then $\LS(G+H)\ge \LSu(G+H)= \LSu(G)+h$.
\end{corollary}

\begin{proof}
By Theorem~\ref{thm:pallowscores} it suffices to show that $\Lsu(H)=\Rsu(H)=h$.
Since Left starts, $\Lsu(H)=h$. Now $\Rsu(H)\le h$, since
Right can achieve the score $h$ by passing. Since $H\in\GS$ then
 $h=\min\{x: \emptyset^x \text{ is an atom in } H\}$. Hence $\Rsu(H)=h$.
 That  $\LS(G+H)\ge \LSu(G+H)$ is by definition. 
\end{proof}

    \begin{definition}
    Let $s\in \mathbb{R}$ and $G\in \GS$.
   The game $G$ is \textit{left-$s$-protected} if  $\LSu(G)\ge  s$
   and either
    $G$ is right-atomic
    or for all    $G^R$, there exists $G^{RL}$
     such that $G^{RL}$ is left-$s$-protected. Similarly,
$G$ is right-$s$-protected if
     $\RSo(G)\le  s$ and, for all
     $G^{L}$, there exists $G^{LR}$
     such that $G^{LR}$ is right-$s$-protected.
    \end{definition}
    In \cite{LarssNS} we prove a necessary and sufficient condition for a game to be greater than or equal to a number.
\begin{theorem}[A Generalized Ettinger's Theorem \cite{LarssNS}]\label{thm:ettinger}
 Let $s\in \mathbb{R}$ and $G\in \GS$. Then $G\su  s$ if and only if $G$ is left-$s$-protected.\end{theorem}

%%%%%%%%%%%%%%%%%%%%%%%%%
%%%%%%%%%%%%%%%%%%%%%%%%%
\section{Reductions  and Canonical Form}\label{sec:reductions}

%%%%%%%%%%%%%%%%%%%%%%%%%
%%%%%%%%%%%%%%%%%%%%%%%%%

 The reduction results, Theorems \ref{thm:domination}, \ref{thm:nonatomic}, and \ref{thm:atomic},
  give conditions under which the options of a game can be modified resulting
  in a game in the same equivalence class. In all cases, it 
  is easy to check that the new game is also in $\GS$. Theorem \ref{thm:substitute}
   requires an explicit check that the modified game is a guaranteed game. In Normal-play games, the reduction procedures result in a 
   unique game, which also has minimum birthday, called the `canonical form'.  
   It is noted by Johnson that both the scoring games he studied and those studied by Ettinger
   there may be many equivalent games with the minimum birthday. The same is true for guaranteed games. 
   However, Theorem \ref{thm:substitute} gives a reduction that while it does not necessarily reduce the birthday
   does lead to a unique reduced game. 

 The results in this section will often involve showing that $G\su H$ or $G\sim H$
  for some games $G$, $H$ where both have the same right options and they differ only slightly in the left options.
  Strategically, one would believe that only the non-common left options need to be considered in inductive proofs,
  that is, the positions of $(\GL\setminus \HL)\cup(\HL\setminus \GL)$.
  The next lemma shows that this is true. 
  
\begin{lemma}\label{lem:rightscores}
Let $F$ and $K$ be guaranteed games with the same sets of right options, and in case this set is empty, the atoms are 
identical. Let $X$ be a guaranteed game. 
\begin{enumerate}
\item  If $\LS(F+X^R)= \LS(K+X^R)$ for all $X^R\in\XR$ then $\RS(F+X)= \RS(K+X)$. 

\item If $\RS(F+X^L) \ge \RS(K+X^L)$, for all $X^L\in\XL$, and $\RS(F^L+X) = \LS(F+X)$, for some 
$F^L\in F^\mathcal{L}\cap K^\mathcal{L}$, 
then $\LS(F+X) \ge \LS(K+X)$.

\end{enumerate}
\end{lemma}

\begin{proof} 
\textit{Part 1:} 
We prove the `$\ge$' inequality and then `$\le$'  follows by symmetry. If Right's best move in $F+X$ 
is obtained in the $X$ component, then $\RS(F+X)=\LS(F+X^R)\ge \LS(K+X^R)\ge \min\{\LS((K+X)^R)\}=\RS(K+X)$. Otherwise, if Right's best move is in the $F$ component, then he 
achieve a score at least as good in $K+X$ by mimicking. 
 If there are no right-options in $F+X$ then neither are there any in $K+X$. Then, by assumption, the right-atom in $F+X$ is identical to the right-atom in $K+X$, and hence the Right-stops are identical. 

The proof of part 2 is very similar to that of part 1, since the respective Right-stops are obtained 
via a common option.
\end{proof}

For example, in part 2 of Lemma~\ref{lem:rightscores}, if $\RS(F^L+X) = \LS(F+X)$, for some $F^L\in F^\mathcal{L}\setminus K^\mathcal{L}$, then the inequality $\LS(F+X) \ge \LS(K+X)$ does not follow directly. As we will see later in this section, when it holds, it is by some other property of the games $F$ and $K$. 

The next result re-affirms that provided a player has at least one option then adding another option
cannot do any harm. This is not true if the player has no options. For example,
consider $G=\langle\emptyset^1\mid 2\rangle$, now adding the left option $-1$ to $G$ gives the game
 $H=\langle -1\mid 2\rangle$.
 But, since $\LS(G) = 1$ and $\LS(H) =0$ then $H\not \su G$.

\begin{lemma}\label{lem:greediness} (Monotone Principle)\\
Let $G\in{\GS}$. If  $|\GL|\ge 1$  then for any $A\in\GS$,
$ \langle \GL\cup A\mid {\GR}\rangle\su G$.
\end{lemma}

\begin{proof} The proof is clear since Left never has to use the new option.
\end{proof}
\subsection{Reductions}\label{sec:3reductions}

We first consider the most straightforward reduction, that of removing dominated options. For this to be possible we require at least two left options.

\begin{theorem}\label{thm:domination} (Domination)
Let $G\in{\GS}$ and suppose $A, B\in \GL$ with $A\pr B$. Let $H=\langle\GL\setminus\{A\}\mid \GR\rangle$.
Then $H\in \GS$ and $G \sim H$.
\end{theorem}

\begin{proof}
\noindent
Note that $H\in \GS$, because $H$ is not atomic (at least $B$ is a left option) and $G\in \GS$. By the monotone principle, Lemma
\ref{lem:greediness}, $G\su H$. Therefore we only have to prove that $H\su G$.
For this, we need to show that $ \LS(H+X)\geqslant \LS(G+X)$ and $\RS(H+X)\geqslant \RS(G+X)$ for all $X$. We will
proceed by induction on the birthday of $X$. Fix $X\in \GS$.

By induction,  for each $X^R\in\XR$, we know that
$\LS(H+X^R)\ge \LS(G+X^R)$. Thus from Lemma~\ref{lem:rightscores}(1), 
it follows that $\RS(H+X)\ge \RS(G+X)$.

Now consider the Left-stops. By induction,  for each $X^L\in\XL$, we know that
$\RS(H+X^L)\ge \RS(G+X^L)$, that is the first condition of Lemma
\ref{lem:rightscores}(2) is satisfied.  By assumption, the only non-common
 option is $A\in G\setminus H$. Therefore,  by Lemma~\ref{lem:rightscores}(2), it suffices to study the case  
 $\LS(G+X) = \RS(A+X)$. Since $A\pr B$, we get $\LS(H+X)\ge \RS(B+X)\ge\RS(A+X) = \LS(G+X)$. Hence $H\su G$, and so $H\sim G$.
\end{proof}

We remind the reader that while we only define the following concepts from Left's perspective,
the corresponding Right concepts are defined analogously.

\begin{definition}
For a game $G$, suppose there are followers $A\in \GL$ and
$B\in A^\mathcal R$ with $B\pr G$.
 Then the Left option $A$ is \textit{reversible}, and sometimes, to be specific, $A$
 is said to be \textit{reversible  through} its right option $B$. In addition, $B$ is called a \textit{reversing} option for $A$ and,
 if $B^\mathcal{L}$ is non-empty then $B^\mathcal{L}$
  is a \textit{replacement set} for $A$. In this case, 
 $A$ is said to be \textit{non-atomic-reversible}.
 If the reversing option is left-atomic, that is, if $B^\mathcal{L}=\emptyset^\ell$, then $A$ is
 said to be \textit{atomic-reversible}.
 \end{definition}

If Left were to play a reversible option then Right has a move that retains or improves his situation. 
Indeed,
it is the basis for the second reduction. In Normal-play games,
\textit{bypassing a reversible option} is to replace a reversible option by
its replacement set, even if the replacement set
is empty. This results in a simpler game equal to the original.
 In $\GS$, there are more cases to consider. 
 %RJN
 We begin by showing that, if the replacement set is non-empty, 
 then bypassing a reversible option does result in a new but equal game. 
In Theorem~\ref{thm:atomic}, we then treat the case of an atomic-reversible option.

\begin{theorem}[Reversibility 1]\label{thm:nonatomic} 
Let $G\in{\GS}$ and suppose that $A$ is a left option of $G$ reversible through $B$.
If  $B^\mathcal{L}$ is non-empty, then $G\sim\left\langle \GL\setminus\{A\},B^\mathcal{L} \mid {\GR}\right\rangle$.
\end{theorem}

\begin{proof}

Consider $G, A, B$ as in the statement of the theorem, and recall that, since
$B$ is a reversing right option, $G\su B$. Moreover, there is a replacement set
$B^\mathcal{L}$, so we let
$H = \left\langle \GL\setminus\{A\},B^\mathcal{L} \mid {\GR}\right\rangle$.
 We need to prove that $H\sim G$, i.e., $\LS(G+X)=\LS(H+X)$ and
 $\RS(G+X)=\RS(H+X)$ for all $X$. We proceed by induction on the birthday of $X$.

 Fix $X$.
 Note that $B^\mathcal{L}$, $\GL$ and $\HL$ are non-empty so that $B+X$, $G+X$
 and $H+X$ all have Left options. Moreover $A+X$ has Right options.

 For the Right-stops: by induction we have that $\LS(G+X^R)=\LS(H+X^R)$ for any 
 $X^R\in\XR$. Thus by Lemma~\ref{lem:rightscores}(1), we have $\RS(G+X)=\RS(H+X)$.

 For the Left-stops, and within the induction, we first prove a necessary inequality.\\
 
\noindent \textit{Claim~1:}  $H\su B$.\\

\noindent \textit{Proof of Claim~1:}
 For the Left-stops:  if $C\in B^\mathcal{L}$ then $C\in\HL$ and thus
 $\LS(H+X)\ge\RS(C+X)$. If $\LS(B+X)=\RS(C+X)$ for some $C\in B^\mathcal{L}$
  then it follows that
  $\LS(H+X)\ge\LS(B+X)$. Otherwise, $\LS(B+X)=\RS(B+X^L)$. By induction,
  $\RS(B+X^L)\le\RS(H+X^L)$ and since $\RS(H+X^L)\le\LS(H+X)$, we get $\LS(H+X)\ge \LS(B+X)$.

For the Right-stops: by the argument before the claim,  $\Rs(H+X)= \Rs(G+X)$.
Since $G\su B$ then $\RS(G+X)\ge \Rs(B+X)$ and thus $\RS(H+X)\ge \Rs(B+X)$.
 This concludes the proof of Claim~1.\\

By induction we have that $\RS(G+X^L)=\RS(H+X^L)$ for any 
 $X^L\in\XL$, which gives the first assumptions of Lemma~\ref{lem:rightscores}(2). 
 It remains to consider the cases where the
  second assumption does not hold.

First, we consider $\LS(G+X)$.
By Lemma~\ref{lem:rightscores}(2), the remaining case to consider is $\LS(G+X) = \RS(A+X)$.
Since $B\in A^\mathcal{R}$, we have
$\RS(A+X)\le \LS(B+X)$. By Claim~1, we know that
$\LS(H+X)\ge \LS(B+X)$. By combining these inequalities we obtain $\Ls(G+X)\le \Ls(H+X)$.

Secondly, we consider $\LS(H+X)$. The only possibly non-common option is $C\in B^\mathcal{L}$,  with 
$C\in \HL\setminus \GL$, and where we, by  Lemma~\ref{lem:rightscores}(2), may assume that $\Ls(H+X)=\Rs(C+X)$. 
Moreover, $G\su B$, and thus $\Ls(H+X)=\Rs(C+X)\le \Ls(B+X)\le \Ls(G+X)$.
\end{proof}

For the next reduction theorem, there is no replacement set, because the reversing option
is left-atomic.  We first prove a strategic fact about atomic reversible
options---nobody wants to play to one!

\begin{lemma}[Weak Avoidance Property]\label{lem:WeakAvoidanceProperty}
Let $G\in{\GS}$ and let $A$ be an atomic-reversible Left option of $G$.
For any game $X$,
if $\XL\ne \emptyset$ then there
is an $X^L$ such that $\RS(A+X)\leqslant \RS(G+X^L)$.
\end{lemma}

\begin{proof}
Let $A$ be an atomic-reversible Left option of $G$ and let $B\in A^\mathcal{R}$ be
a reversing option for $A$. Assume that $X$ has a left option.

 By definition, $G\su B$
and $B=\left\langle \emptyset^\ell\mid B^\mathcal{R}\right\rangle$.
 Since $B$ is a right option of $A$ then
 $A+X\ne \langle (A+X)^\mathcal{L}\mid \emptyset^r\rangle$. Consequently,
\begin{eqnarray*}
\RS(A+X)&\leqslant& \LS(B+X)  \\
&=&   \RS(B+X^L),
        \mbox{ for some $X^L$,}\\
&\leqslant&
  \RS(G+X^L), \mbox{ since $G\su B$}.
\end{eqnarray*}
\end{proof}

The next reduction is about replacing a left atomic-reversible option $A$ in a game $G$.
There are two cases. If Left has a `good' move other than $A$ then $A$ can be eliminated.
Otherwise, we can only simplify $A$.

    \begin{theorem}[Atomic Reversibility]\label{thm:atomic} 
Let $G\in{\GS}$ and suppose that $A\in\GL$ is reversible through $B=\langle\emptyset^\ell\mid  B^\mathcal{R}\rangle$.
\begin{enumerate}
  \item If $\LSu(G)=\RSu(G^L)$ for some $G^L\in G^{\cal L}\setminus\{A\}$, then
  $G\sim \left\langle \GL\setminus \{A\}\mid {\GR}\right\rangle$;
  \item If $\LSu(G)=\RSu(A)>\RSu(G^L) $  for all $G^L\in \GL\setminus \{ A\}$, then

   $G\sim \left\langle \left\langle\emptyset^\ell\mid  B\right\rangle,\GL\setminus\{A\}\mid {\GR}\right\rangle$.
\end{enumerate}
\end{theorem}

\begin{proof} Let  $A\in \GL$ and $B\in A^\mathcal R$  be as in the
statement of the theorem, with $G\su B$. First an observation:\\

\noindent \textit{Claim 1:} $\LSu(G)\geqslant \ell$.\\

\noindent Let $n$ be the birthday of $G$ and since $B$ is a proper follower of $G$, the birthday of $B$ is less than $n$.
Since $G\su B$, from  Lemma~\ref{lem:waiting moves} we have
\[\LS(G-\Num{n})\geqslant \LS(B-\Num{n}) = \LS(\left\langle \emptyset^\ell\mid B^\mathcal{R}\right\rangle-\Num{n})=\ell, \]
where $n$ is the birthday of $G$. This proves the claim.\\

The proof of the equality in both parts will proceed by induction on the birthday of $X$.
Again, in both parts, let $H$ be the game that  we wish to show is equal to $G$.
We have, by induction,
that $\LS(G+X^R) = \LS(H+X^R)$, and by $\GR=\HR$, from Lemma~\ref{lem:rightscores}(1),
it then follows that $\RS(G+X) = \RS(H+X)$.\\

It remains to show that $\LS(G+X)=\LS(H+X)$ in both parts.\\

\noindent \textit{Part 1.}
The assumption is that there exists $C\in \GL\setminus \{A\}$ with $\LSu(G)=\RSu(C)$.
Let $H=\left\langle \GL\setminus \{A\}\mid {\GR}\right\rangle$. Note that both $G+X$ and $H+X$ have left options
since $C$ is in both $\GL$ and $\HL$.
From Lemma~\ref{lem:greediness} we have $G\su H$,  and thus it remains to show that
$\LS(H+X)\ge \LS(G+X)$.

By Lemma~\ref{lem:rightscores}(2), we need only consider the case
 $\LS(G+X) = \RS(A+X)$. Note that $X$ must be left-atomic; else, by
 Lemma~\ref{lem:WeakAvoidanceProperty}, there would exist $X^L\in\XL$ with
  $\RS(A+X) \le \RS(G+X^L)$. Therefore, we may assume that
  $X=\langle \emptyset^x\!\mid \!X^\mathcal R\rangle$.
  In this case, since $C\ne A$ is the best pass-allowed Left move in $G$ then
   this is also true for
   $H$. We now have the string of inequalities,
$$\LS(H+X)\ge\LSu(H+X)=\LSu(H)+x=\RSu(C)+x = \LSu(G)+x\ge \ell+x,$$
where the first inequalities are from Corollary~\ref{cor:emptywaiting},
 and the last inequality is by Claim~1.
 Since $B$ is a right option of $A$, we also have that
  $$\LS(G+X)=\RS(A+X)\le \LS(B+X)= \ell+x.$$
Thus $\LS(G+X)\le \LS(H+X)$ and this completes the proof of part~1 of the theorem.\\

\noindent \textit{Part 2.}
In this case, the Right's-pass-allowed Left-stop of $G$ is obtained
only through $A$. 
Let $H=\left\langle \left\langle\emptyset^\ell\mid  B\right\rangle,\GL\setminus\{A\}
\mid {\GR}\right\rangle$.
 Recall that it only remains to show that $\LS(G+X)=\LS(H+X)$, and that, by Lemma~\ref{lem:rightscores}, 
 we only need to consider the non-common options in the respective games.

First, suppose   $\LS(H+X) = \RS(\langle\emptyset^\ell\mid  B\rangle+X)$.
  Since $G\su B$ and $B$ is a right option of $\langle\emptyset^\ell\mid  B\rangle$, 
  we have the inequalities
$$\Ls(H+X)=\Rs(\langle \emptyset^\ell \mid B\rangle+X)\le \Ls(B+X)\le \Ls(G+X).$$
Thus, by Lemma~\ref{lem:rightscores}(2), $\LS(H+X)\le\LS(G+X)$.\\

Secondly, suppose that $\LS(G+X) = \RS(A+X)$.
Note that if $X$ has a left option then, by Lemma~\ref{lem:WeakAvoidanceProperty},
there exists some $X^L\in\XL$ such that
$\LS(G+X)=\RS(G+X^L)$. By induction, then $\RS(G+X^L)=\RS(H+X^L)\le \LS(H+X)$.
Therefore, we may assume that $X=\langle \emptyset^\ell \mid \XR \rangle $.
Since $B$ is a right option of $A$, the only Left option in $G$, 
we have the string of inequalities
$$\Ls(G+X)=\Rs(A+X)\le \Ls(B+X)=\ell+x.$$
 To show that $\Ls(H+X)\ge \ell +x$, we note that it suffices for Left to move in the
  $H$ component to $\langle \ell \mid B\rangle \in \HL$,
  since all  scores in $B=\langle \ell \mid B^\mathcal R\rangle$ are at least $\ell$.
  Thus, by Lemma~\ref{lem:rightscores}(2), we now have $\LS(G+X)\le \LS(H+X)$. 
  
  From this, together with the conclusion of the previous paragraph, we have
  $\LS(G+X)=\LS(H+X)$.
\end{proof}

Suppose that $G\in \GS$ has an atomic-reversible option, $A\in \GL$, with the reversing option $B=\langle\emptyset^\ell\mid B^\mathcal{R}\rangle$.
Given the reduction in Theorem~\ref{thm:atomic}(2), a remaining problem of atomic reducibility is to find a simplest substitution for $B$. In Section~\ref{sec:uniqueness}, we will show that the following result solves this problem.

\begin{theorem}[Substitution Theorem]\label{thm:substitute}
Let $A$ be an atomic-reversible Left option of $G\in \GS$ and let $B=\langle\emptyset^\ell\mid  B^\mathcal{R}\rangle$ be a reversing Right option of $A$. 
Suppose also that
$\LSu(G)=\RSu(A)>\RSu(G^L) $  for all $G^L\in \GL\setminus \{ A\}$.
\begin{itemize}
\item[1.] There exists a smallest nonnegative integer $n$ such that $G\su \ell -\widehat {n}$
and $G \sim \langle \ell - (\widehat {n+1}), \GL\setminus\{A\} \mid \GR \rangle $.
\item[2.] If $A$ is the only Left option of $G$ and
$\langle \emptyset^\ell\mid G^\mathcal{R}\rangle\in \GS$,
 then $G\sim \langle \emptyset^\ell\mid G^\mathcal{R}\rangle$.
\end{itemize}
\end{theorem}

\begin{proof} 
{\bf Case 1:} Let $m=b(B)$. By assumption $G\su B$ and, by Theorem~\ref{thm:projection}(2), $B\su \ell-\widehat{m}$, and thus $G\su\ell -\widehat{m}$.
 Since $m$ is a nonnegative integer, the existence part is clear.
 Let $n$ be the minimum nonnegative integer such that $G\su \ell - \widehat {n}$.

Let $K=\ell - (\widehat {n+1})$, which upon expanding is $\langle\emptyset^\ell \! \mid \! \ell - \widehat {n}\rangle$, let
$H= \langle K, \GL\setminus\{A\} \mid \GR \rangle $, and let 
$G' = \langle K, \GL \mid \GR \rangle $.
By Lemma~\ref{lem:greediness} and the definition of $n$, we have
$G'\su G\su \ell - \widehat{n}$. Hence $\ell - \widehat{n}$ is a reversing game
in both $G$ and $G'$, and both $A$ and $K$ are atomic-reversible Left options in $G'$.

Since $G$ satisfies part~2 of Theorem~\ref{thm:atomic}. Then Claim~1 in Theorem~\ref{thm:atomic} can be strengthened.\\

\noindent\textit{Claim 1:}  $\LSu(G)=\ell$.\\

\noindent\textit{Proof of Claim 1:} This is true  because $\LSu(G)=\RSu(A)\leqslant\LSu(B)=\ell$.\\

Hence, $\ell = \Lsu(G) = \Rsu(A)$.
 We also have that $\Rsu(K)= \ell $. It is now easy to see that
$\Lsu(G')=\ell.$ Thus we have two atomic-reversible Left options in $G'$,
and so we can apply part~1 in Theorem~\ref{thm:atomic}. We get that $G'\sim G$ since $K$ is an  atomic-reversible Left option in $G'$. Moreover, $G' \sim H$, since $A$ is also atomic-reversible. This finishes the proof of Case 1.\\

\noindent {\bf Case 2:} This is the case where $\GL = \{A\}$. We put $H=\langle \emptyset^\ell\mid  G^\mathcal{R}\rangle\in \GS$. To prove $G\sim H$ we proceed by induction on the birthday of the distinguishing game $X$. 

 From Lemma~\ref{lem:rightscores}(1) and induction, we have that $\RS(G+X)=\RS(H+X)$, for any $X\in \GS$.

For the Left-stops, from Case 1, we know that 
$G\sim\langle \ell -(\Num{n+1}) \mid \GR \rangle $.
Therefore, in the case $X=\langle \emptyset^x \mid \emptyset^y\rangle$ it is easy to see that $\Ls(H+X)=\ell+x \le \Ls(G+X)$, since $y\ge x$. 
Moreover, we also have $\Ls(G+X)=\Rs(A+X)\le\Ls(B+X)=\ell+x$, which thus proves equality.

If $\XL= \emptyset^x$ and $\XR\ne \emptyset$, then, $\LS(G+X) =\Rs(\ell - (\Num{n+1}) +X)$
 and it is clear that Right can obtain the score $\ell+x$ by playing to $\ell -\Num{n}+X$.
 Since both games are left-atomic and in $\GS$, then $\Rs(\ell - (\Num{n+1}) +X)\ge \ell+x$,
 so in fact, equality holds. Hence, in this case, we get $\Ls(G+X)=\ell+x=\Ls(H+X)$.

If $\XL\ne \emptyset$,  then by Lemma~\ref{lem:WeakAvoidanceProperty} (weak avoidance),
there is some $X^L$ such that $\RS(A+X) \leqslant \RS(G+X^L)$.
Therefore, $\LS(G+X) = \max\{\RS(G+X^L) : X^L\in \XL\}$.
Also, $\LS(H+X) = \max\{\RS(H+X^L) : X^L\in \XL\}$ since there is no
Left move in $H$. By induction, $\RS(H+X^L) = \RS(G+X^L)$ and
consequently, $\LS(G+X)=\LS(H+X)$.
\end{proof}
In summary, there are four types of reductions for Left:
\begin{enumerate}
  \item Erase dominated options;
  \item Reverse non-atomic-reversible options;
  \item Replace atomic-reversible options by $\ell-\widehat{\,n\!+\!1\,}$;
  \item If possible, when an atomic-reversible is the only left option,
   substitute $\emptyset^\ell$ for $\ell-\widehat{\,n\!+\!1\,}$. 
\end{enumerate}
Here $\ell$ is a real number and $n\ge 0$ is an integer (as given in Theorem~\ref{thm:substitute}) providing a number of waiting moves for Right. We have the following definition.

\begin{definition}\label{def:reducedform}
A game  $G\in\GS$ is said to be \textit{reduced} if none of Theorems~\ref{thm:domination},
\ref{thm:nonatomic},~\ref{thm:atomic}, or~\ref{thm:substitute} can be applied to
$G$ to obtain an equivalent game with different sets of options.
\end{definition}
%%%%%%%%%%%%%%%%%%%%%%%%%
%%%%%%%%%%%%%%%%%%%%%%%%%
\subsection{Constructive Game Comparison}\label{sec:gamecomparison}

%%%%%%%%%%%%%%%%%%%%%%%%%
%%%%%%%%%%%%%%%%%%%%%%%%%
We wish to prove that, for a given guaranteed scoring game, there is one unique reduced game representing the full congruence class, a canonical form. 
To this purpose, in this subsection, we first develop another major tool (also to be used in Section~\ref{sec:conjugates}) of constructive game comparison. 
 The existence of a canonical form is far from obvious, as the order of reduction can vary. In Normal-play, the proof of uniqueness uses the
 fact that if $G\sim H$ then $G-H\sim 0$. However, in (guaranteed) scoring play, 
 $G\sim H$ does \emph{not} imply $G+\ConjS{H}\sim 0$.
 We use an idea, `linked', adapted from Siegel\cite{Siege2015}, which only uses the partial order. 
 To fully adapt it for guaranteed games, we require a generalization of Theorem~\ref{thm:ettinger}
 (which in its turn is a generalization of Ettinger's\cite{Ettin1996} theorem for dicot games).

Recall that  $\Conj{G}=\langle \ConjS{\GR} \mid \ConjS{\GL} \rangle$, where the conjugate is applied to the respective options, and
if, for example, $\GR =\atom{r}$, then $\ConjS{\GR}=\atom{-r}$. 

\begin{definition}\label{def:adjoint}
Let $G\in \GS$ and let $m(G) = \max\{|t|: \emptyset^t \mbox{ is an atom in $G$}\}$.
Let $r, s$ be two nonnegative real numbers. 
The \emph{$(r,s)$-adjoint} of $G$ (or just \emph{adjoint}) is $G^\circ_{r,s}= \Conj{G}+\langle \emptyset^{-m(G)-r-1}\mid \emptyset^{m(G)+s+1}\rangle $.
\end{definition}

Since $-m(G)-r-1\le m(G)+s+1$, it follows that $G^\circ_{r,s}\in\GB$.

\begin{theorem}\label{thm:adjoint}
Given $G\in\mathbb{GS}$ and two nonnegative real numbers $r,s$  then

\noindent $\LS(G+G^\circ_{r,s})<-r$ and $\RS(G+G^\circ_{r,s})>s$.
\end{theorem}

\begin{proof}
In the game $G+\Conj{G}+\langle \emptyset^{-m(G)-r-1}\mid \emptyset^{m(G)+s+1}\rangle$,
 the second player can mirror each move in the $G +\Conj{G}$ component,
and there are no other moves since the remaining component is purely-atomic.
Therefore, $$\LS(G + G^\circ_{r,s}) = \LS(G+\Conj{G})-m(G)-r-1 \le m(G)-m(G)-r-1<-r.$$
The bound for the Right-stop is obtained similarly.
\end{proof}
\begin{observation}\label{obs:adjoint}
If $r=s=0$ in Definition~\ref{def:adjoint}, then Theorem~\ref{thm:adjoint} corresponds to the particular case where $\LS(G+G^\circ_{0,0})<0$ and $\RS(G+G^\circ_{0,0})>0$. This will suffice in the below proof of Lemma~\ref{lem:linked}. Thus we will use the somewhat simpler notation $G^\circ$ for the $(0,0)$-adjoint of $G$.
\end{observation}
\begin{definition}
Let $G, H\in\mathbb{GS}$. We say that \emph{$H$ is linked to $G$ (by $T$)} if there exists some $T\in\mathbb{GS}$ such that $\Ls(H+T) < 0 < \Rs(G+T)$.
\end{definition}
Note that, if $H$ is linked to $G$, it is not necessarily true that $G$ is linked to $H$.
%RJN
\begin{lemma}\label{dl1}
Let $G, H\in\mathbb{GS}$. If $H\su G$ then $H$ is linked to no $G^L$ and no $H^R$ is linked to $G$.
\end{lemma}

\begin{proof}
Consider $T\in\GS$ such that $\LS(H+T)<0$. Because $H\su G$,
we have $\LS(H+T)\ge \LS(G+T)$.
Therefore, $0>\LS(H+T)\ge \LS(G+T)\ge \RS(G^L+T)$, for any $G^L$.
Analogously, consider  $T\in\GS$ such that $0<\RS(G+T)$, we have $0<\RS(G+T)\le \RS(H+T)\le \LS(H^R+T),$ for any $H^R$.
\end{proof}

\begin{lemma}\label{dl2}
Let $G,H\in\mathbb{GS}$. Suppose that $G\not\succcurlyeq H$.
\begin{enumerate}
  \item There exists $X\in \GS$ such that $\LS(G+X)<0<\LS(H+X)$
  \item There exists $Y\in \GS$ such that $\RS(G+Y)<0<\RS(H+Y)$.
\end{enumerate}
\end{lemma}

\begin{proof}
By assumption, there exists $X$ such that $\LS(G+X) < \LS(H+X)$ \emph{or} there exists $Y$ such that $\RS(G+Y)<\RS(H+Y)$.
By Theorem~\ref{thm:LRscores} (the claim in its proof), we have that
$$\exists X:\LS(G+X)<\LS(H+X) \Leftrightarrow \exists Y:\RS(G+Y)<\Rs(H+Y).$$
Suppose that there exists $Z$ such that $\alpha = \Ls(G+Z)<\Ls(H+Z)=\beta$.
Let $X=Z - (\alpha+\beta)/2$. Then
$\Ls(G+X) = \Ls(G+Z) - (\alpha+\beta)/2=(\alpha-\beta)/2<0$ and
 $0<(\beta-\alpha)/2=\Ls(H+Z)- (\alpha+\beta)/2=\Ls(H+X)$.
 Hence the first part holds. The proof of the other part is analogous.
\end{proof}

\begin{lemma}\label{lem:linked}
Let $G,H\in\mathbb{GS}$.
Then $G$ is linked to $H$ if and only if no $G^L\su H$ and
no $H^R\pr G$.
\end{lemma}

\begin{proof}
\noindent
($\Rightarrow$): Consider $G$ linked to $H$ by $T$, that is $\LS(G+T)<0<\RS(H+T)$. It follows
\begin{enumerate}
  \item $\RS(G^L+T)\leqslant \LS(G+T)<0<\RS(H+T)$, for any $G^L$
  \item $\LS(G+T)<0<\RS(H+T)\leqslant \LS(H^R+T)$, for any $H^R$.
\end{enumerate}

\noindent
The two items contradict both $G^L\su H$ and $H^R\pr G$.\\

\noindent
($\Leftarrow$): Suppose no $G^L\su H$ and no $H^R\pr G$. Consider
$G^\mathcal{L}=\{G^{L_1},\ldots,G^{L_k}\}$ and
$H^\mathcal{R}=\{H^{R_1},\ldots,H^{R_\ell}\}$, including the case that either or both are atoms. 
By Lemma~\ref{dl2}, for each $i$, $1\le i\le k$,
we can define $X_i$ such that $\LS(G^{L_i}+X_i) < 0 < \LS(H+X_i)$, and,
for each $j$, $1\le j\le \ell$, we can define  $Y_j$ such that $\RS(G+Y_j) < 0 < \RS(H^{R_j}+Y_j)$.
 Let $T=\langle \, T^\mathcal{L}\mid T^\mathcal{R}\,\rangle$ where
 $$T^\mathcal{L}=\left\{\begin{array}{ll}
                         \{-g-1\}, & \text{if} \,\,G = \langle \GL\!\mid \!\emptyset^g\rangle \,\,\text{and also}\,\,H\,\, \text{is right-atomic}; \\
                              {G^\mathcal{R}}^\circ\bigcup\cup_{j=1}^{\ell}\{Y_i\}, & \text{otherwise}.
                                                     \end{array}
\right.$$

$$T^\mathcal{R}=\left\{\begin{array}{ll}
                         \{-h+1\}, & \text{if}\,\,H=\langle \emptyset^h\!\mid \!\HR\rangle\,\,\text{and also}\,\,G\,\, \text{is left-atomic} ;\\
                          {H^\mathcal{L}}^\circ\bigcup\cup_{i=1}^{k}\{X_i\}, & \text{otherwise}.
                                                 \end{array}
\right.$$
Here ${G^\mathcal{R}}^\circ$ denotes the set of $(0,0)$-adjoints of the Right options of $G$, and if there is no Right option of $G$, then it is defined as the empty set. Note that, in this case, if also $\HR$ is empty, then the first line of the definition of $\mathcal T^\mathcal L$ applies, so $\mathcal T^\mathcal L$  (and symmetrically for $\mathcal T^\mathcal R$) is never empty.   
Thus $T\in \GS$, because each option is a guaranteed game. (For example, if both $G=\pura{a}{b}$ and $H=\pura{c}{d}$ are purely-atomic guaranteed games, then $T=\langle -b-1\mid -c+1\rangle$ is trivially guaranteed, because each player has an option to a number. Note also that the scores $a$ and $d$ become irrelevant in this construction.) 

Consider first $G+T$ with $T^\mathcal{L}$ as in the second line of the definition. It follows that
 $\LS(G+T)<0$ because:
\begin{enumerate}
  \item if Left plays to $G^{L_i}+T$, then, because there is a Left option, the second line applies also to $T^\mathcal{R}$. Right answers with $G^{L_i}+X_i$, and $\LS(G^{L_i}+X_i)<0$, by definition of $X_i$;

  \item if Left plays to $G+{G^R}^\circ$, Right answers in $G$ to the corresponding $G^R$
   and $\LS(G^R+{G^R}^\circ)<0$, by Observation~\ref{obs:adjoint};

  \item if Left plays to $G+Y_i$, then by construction, $\RS(G+Y_i)<0$.
\end{enumerate}

Consider next $G+T$ with $T^\mathcal{L}$ in the first line of the definition. We get that $\LS(G+T)<0$ because, either
\begin{enumerate}
  \item $\LS(G+T)=\RS(G+T^L)=\RS(G-g-1)=g-g-1=-1<0$; or
  \item $\LS(G+T)=\RS(G^{L_i}+T)\leqslant\LS(G^{L_i}+X_i)<0$.
\end{enumerate}

\noindent
The last case follows because there are left options in $G$, so the second line
 of the definition of $T^\mathcal{R}$ applies.
 In every case, $\LS(G+T)<0$. The argument for $\RS(H+T)>0$ is analogous. Therefore, $\LS(G+T)<0<\RS(H+T)$ and $G$ is linked to $H$ by $T$.\end{proof}

In the following result we extend Theorem~\ref{thm:ettinger} 
 by using the linked results. From an algorithmic point of view,  when comparing games
 $G$ and $H$, it ultimately removes the need to consider $G+X$ and $H+X$ for all $X$.
 \footnote{This is as close as guaranteed games get to the Normal-play constructive comparison---Left wins playing second in $G-H$ iff $G\su H$. For not-necessarily-guaranteed scoring games, no efficient method for game comparison is known.}

\begin{theorem}[Constructive Comparison] \label{thm:comparison}
Let $G,H\in\mathbb{GS}$. Then, $G\su H$ if and only if
\begin{enumerate}
\item $\underline{Ls}(G)\geqslant \underline{Ls}(H)$ and $\overline{Rs}(G)\geqslant \overline{Rs}(H)$;
  \item For all $H^L\in H^{\mathcal{L}}, \mbox{ either }
  \exists G^L\in G^{\mathcal{L}}\!:G^L\su H^L\,\,or\,\,\exists H^{LR}\in H^{L\mathcal{R}}\!:G\su H^{LR};$
  \item For all $G^R\in G^{\mathcal{R}}, \mbox{ either }\exists H^R\in H^{\mathcal{R}}\!:G^R\su H^R\,\,or\,\,\exists G^{RL}\in G^{R\mathcal{L}}\!:G^{RL}\su H.$
\end{enumerate}
\end{theorem}

\begin{proof}
\noindent
($\Rightarrow$)
Suppose that $\underline{Ls}(G)< \underline{Ls}(H)$. Then, for some $n$, $\LS(G-\widehat{n})<\LS(H-\widehat{n})$. 
This, however, contradicts $G\su H$ and so part~1 holds.
  
  Consider $H^L\in H^{\mathcal{L}}$. Because $G\su H$ , by Lemma~\ref{dl1}, 
  $G$ is not linked to $H^L$. Therefore, by Lemma~\ref{lem:linked}, we have 
  $\exists G^L\in G^{\mathcal{L}}:G^L\su H^L\,\,or\,\,\exists H^{LR}\in H^{L\mathcal{R}}:G\su H^{LR}$. 
  The proof of part~3 is similar.
  
($\Leftarrow$)
Assume $1$, $2$ and $3$, and also suppose that $G \not \su H$. By the definition of the partial
order, there is a distinguishing game $X$ such that either $\Ls(G+X)<\Ls(H+X)$ or $\Rs(G+X)<\Rs(H+X)$.
Choose $X$ to be of the smallest birthday such that $\Ls(G+X)<\Ls(H+X)$. There are three cases:

\begin{enumerate}[(a)]

 \item $H+X=\langle \emptyset^h\mid H^{\mathcal{R}}\rangle+\langle \emptyset^x\mid X^{\mathcal{R}}\rangle$.

\noindent In this case, $\Ls(H+X)=h+x$. On the other hand,
$\Ls(G+X)\geqslant \underline{Ls}(G+X)\geqslant \underline{Ls}(G)+\underline{Rs}(X)$
(this last inequality holds by Theorem~\ref{thm:pallowscores}).
Also, $\underline{Ls}(G)+\underline{Rs}(X)\geqslant \underline{Ls}(H)+x$,
because $\underline{Ls}(G)\geqslant \underline{Ls}(H)$ and by $X\in\GS$, Definition~\ref{def:guaranteed}(2).
 Finally, $\underline{Ls}(H)+x=h+x$ because $\underline{Ls}(H)$ is trivially equal to $h$. This contradicts $\Ls(G+X)<\Ls(H+X)$.
 
  \item $\Ls(H+X)=\Rs(H^L+X)$, for some $H^L\in H^{\mathcal{L}}$.

  \noindent In this case, because of part $2$, we have either $G^L\su H^L$ or $G\su H^{LR}$. If the first holds, 
then $\Ls(G+X)\ge \Rs(G^L+X)\ge \Rs(H^L+X)=\Ls(H+X)$. If the second holds, then 
 $\Ls(G+X)\geqslant \Ls(H^{LR}+X)\geqslant \Rs(H^L+X)=\Ls(H+X)$. Both
 contradict the assumption $\Ls(G+X)<\Ls(H+X)$.
 
  \item $\Ls(H+X)=\Rs(H+X^L)$, for some $X^L\in X^{\mathcal{L}}$.

\noindent By the ``smallest birthday'' assumption, $\Rs(G+X^L)\geqslant \Rs(H+X^L)$. Therefore, $\Ls(G+X)\geqslant \Rs(G+X^L)\geqslant \Rs(H+X^L)=\Ls(H+X)$. Once more, we contradict $\Ls(G+X)<\Ls(H+X)$.
 
\end{enumerate}

\noindent
For the Right-stops $\Rs(G+X)<\Rs(H+X)$ the argument is similar.
Hence, we have shown that $G\su H$.
\end{proof}

Note that we can derive the known result, Theorem~\ref{thm:ettinger}, as a simple corollary of Theorem~\ref{thm:comparison}, by letting $H=s$ be a number.
\subsection{Uniqueness of Reduced Forms}\label{sec:uniqueness}

We are now able to prove the existence of a unique reduced form for a congruence class of games. We let $\eqsim$ denote ``identical to'', that is if $G, H \in \GS$, then $G\eqsim H$ if they have identical game tree structure and, given this structure, each atom in $G$ corresponds to an identical atom, in precisely the same position, in the game $H$.

\begin{theorem}\label{thm:uniqred}
Let $G,H\in\mathbb{GS}$. If $G\sim H$ and both are reduced games, then $G\eqsim H$.
\end{theorem}

\begin{proof}

We will proceed by induction on the sum of the birthdays of $G$ and $H$.
We will exhibit a correspondence $G^{L_i}\sim H^{L_i}$ and $G^{R_j}\sim H^{R_j}$
between the options of $G$ and $H$.
By induction, it will follow that $G^{L_i}\eqsim H^{L_i}$, for all $i$,  and $G^{R_j}\eqsim H^{R_j}$, for all $j$, 
and consequently  $G\eqsim H$.\\

\noindent {\it Part 1.}
For the base case, if $G =\pura{a}{b} $ and $H =\pura{c}{d}$ then, 
since $G\sim H$, we must in particular have $a=\LS(G)=\LS(H)=c$ and $b=\RS(G)=\RS(H)=d$. Hence $G\eqsim H$.

Without loss of generality, we may assume that there is a Left option $H^L$. We also assume that if $H^L$ is reversible, then, since $H$ is reduced, it has to be atomic-reversible of the form in Theorem~\ref{thm:substitute}.\\

\noindent {\it Part 2.} Assume that $H^L$ is not atomic-reversible.

Since $G \sim H$, of course, $G\su H$. From Theorem~\ref{thm:comparison},
there exists a $G^L$ with $G^L\su H^L$ or there exists a $H^{LR}\pr G$.
Now $H^{LR}\pr G\sim H$ would contradict that $H^L$ is not reversible.
Thus, there is some $G^L$ with $G^L\su H^L$.

Suppose that $G^L$ is atomic-reversible,  that is,
$G^L \sim \langle\emptyset^\ell \mid \ell-\Num{n}\rangle \sim \ell-\Num{\,n\!+\!1\,}$
for some nonnegative integer $n$ and with $G\su \ell - \Num{n}$.
Since $G\sim H$ we also have $H\su  \ell - \Num{n}$.
(For any real number $s$ and nonnegative integer $m$,
$s-\Num{m}$ is invertible since $s-\Num{m}+(-s+\Num{m}) = 0$.)
Therefore
\[G^L\su H^L\Leftrightarrow \ell - \Num{\,n\!+\!1\,}\su H^L \Leftrightarrow 0\su  H^L - \ell+ \Num{\,n\!+\!1\,},\]
where the last equivalence is by Lemma~\ref{lem:dis}.
From Theorem~\ref{thm:ettinger}, after a Left move from $H^L - \ell+\Num{n+1}$
 to $H^L - \ell+\Num{n}$ , Right must have a move to a position
less than or equal to zero, say $H^{LR}  - \ell + \Num{n} \pr   0$.
The inequalities
$H\su \ell - \Num{n}$ and $ H^{LR}-\ell + \Num{n}\pr   0$ give that $H\su H^{LR}$, which contradicts that $H^L$ is not atomic-reversible. It follows therefore, that $G^L$ is not atomic-reversible.

A similar argument for $G^L$ gives a left option $H^{L'}$ such that $H^{L'}\su G^L$.
Therefore, $H^{L'}\!\su G^L\!\su H^L$. Since there is no domination,
$H^{L'}\!\sim H^L\!\sim G^L$. By induction, $H^L\eqsim G^L$.

The symmetric argument gives that each non-atomic option $H^R$ is identical to some
$G^R$. In conclusion, we have a \emph{pairwise correspondence} between
 options of $G$ and $H$ that are not atomic-reversible.\\

\noindent
{\it Part 3.} Assume that $A=H^L$ is atomic-reversible. 

The proof is divided into two cases.
\\

\noindent
{\bf Case 1:} $|\HL|>1$.

Observe that part 2 of Theorem~\ref{thm:atomic} (the atomic-reversibility theorem) applies,
 because if $A$ would have been as in part 1 of that theorem, then it would have reversed out
 (contradicting the assumptions on $G$ and $H$).
 Therefore, $A$ is the only Left option with $\LSu(H) = \RSu(A)$.

 If, for every $G^L\in \GL$ we have $\LSu(H)  \ne \RSu(G^L)$, 
then  $\LSu(G)\ne \LSu(H)$, which contradicts $G\sim H$. Thus, there is some $A'\in \GL$ with
$\LSu(H)  = \RSu(A')$ and, from the pairwise correspondence
for non-atomic-reversible options, it also follows that $A'$ is atomic-reversible. 
Therefore, we may assume that $A=a-\Num{n+1}$ and that
$A'=a'-\Num{m+1}$ for some real numbers $a, a'$, and some nonnegative integers, $n ,m$. Since $\RSu(A') = \RSu(A)$ then $a=a'$.
That $m=n$ follows from (Theorem~\ref{thm:substitute}(1)), the definition of minimal nonnegative integer, since $A^R=a-\Num{n}$ and $A'^{R}=a'-\Num{m}$ are reversing options. Therefore $A\eqsim A'$, and again, if there was another Left option, $G^L\in \GL$ with $\LSu(G) = \RSu(G^L)$, then it must have been reversed out, because of the assumption of reduced form. Hence $A'$ is the only such Left option in $G$.\\

\noindent
{\bf Case 2:} The only left option of $H$ is $A=\langle \emptyset^h\mid h-\widehat n \rangle$, for some real number $h$ and nonnegative integer $n$, 
that is $H=\langle \langle \emptyset^h\mid h-\widehat n \rangle\mid \HR\rangle$. Since $H$ cannot be reduced further, by the second part of 
Theorem \ref{thm:substitute}, it follows that $\langle \emptyset^h\mid \HR\rangle\not\in\GS$. Thus there must exist an $s$-atom, with $s < h$,
in an atomic follower of $\HR$.

Consider the Left options of $G$. By the pairwise correspondence of non-atomic-reversible options, since $\HL$ has none then neither has $\GL$. 
So, if $\GL$ has options they are atomic-reversible. 

First, suppose that $G=\langle\emptyset^h\mid \GR\rangle$. 

The non-atomic-reversible right options of $G$ and $H$ are paired 
(the conclusion of Part 2 of this proof). Since $G\in \GS$ then $\emptyset^s$ is not in any non-atomic-reversible right option of $G$
and hence $\emptyset^s$ is not in any non-atomic-reversible right option of $H$.
Thus,  either $\HR =\emptyset^s$ or $H$ has a right  atomic-reversible option $\langle s-\Num{m}\mid \emptyset^s\rangle$. In the latter case,
by Theorem \ref{thm:atomic}(2) (with Left and Right interchanged) $\Rs(H) = s$. Thus, in both cases, $\Rs(H) =s$, from which it follows that $\Rs(G)=s$ which, in turn, implies 
that $\emptyset^s$ is in  $\GR$. This again contradicts $G\in\GS$.
Therefore, $G=\langle\emptyset^h\mid \GR\rangle$
is impossible.

Therefore $G=\langle\langle \emptyset^\ell\mid \ell-\widehat m \rangle\mid \GR\rangle$ for some $\ell$ and $m$. 
Since $\Ls(G)=\Ls(H)$ it follows that $\ell = h$.
By Theorem \ref{thm:substitute}, since $G\sim H$, the number of waiting moves (for Right), is given by exactly the same definition as for $H$.  Hence, 
$m=n$ and $\GL=\{A\}$.
\\

 In all cases, we have shown that $\HL$ is identical to $\GL$. The proof for $\HR$ and $\GR$ is similar. Consequently $G\eqsim H$.
\end{proof}

The next result is immediate. It allows us to talk about
\textit{the canonical form} of a game/congruence class.

\begin{corollary}[Canonical Form]
Let $G\in\GS$. There is a unique reduced form of~$G$.
\end{corollary}

Finally, the canonical form can be used for induction proofs since it has 
the minimum birthday of all games in its congruence class. Incidentally, it has the 
least width  (number of leaves) of all such trees. However, minimum birthday and minimum width  is not a characterization
of canonical form. Let  
$G=\langle -1,1-\Num{1}\mid \langle2\mid 2\rangle\rangle$ and
 $H=\langle -1,\langle \emptyset^1\mid \langle \emptyset^1\mid \emptyset^2\rangle\rangle\mid \langle2\mid 2\rangle\rangle$. 
 Then $G$ is the canonical form of $H$ but the game trees of the two games have the same depth and width.

%%%%%%%%%%%%%%%%%%%%%%%%%%%%%%%%%%%%%
\section{Additive Inverses and Conjugates}\label{sec:conjugates}
%%%%%%%%%%%%%%%%%%%%%%%%%%%%%%%%%%%%%%
From the work in Mis\`ere games comes the following concept.

\begin{definition}Let $\mathcal{X}$ be a class of combinatorial games with defined disjunctive sum and game comparison. It has the \textit{Conjugate Property} if for each game $G\in \mathcal{X}$ for which there exists an inverse, that is a game $H\in \mathcal{X}$ such that $G+H\sim 0$, then $H=\Conj{G}$.
\end{definition}

\begin{theorem}\label{Conjugate}$\GB$ has the Conjugate Property.
\end{theorem}

\begin{proof}

Consider $G,H\in \GB$ in their reduced forms, such that $G+H\sim 0$.
We will prove, by induction on the birthday of $G+H$, that we must have $H= \Conj{G}$.\\

\noindent
\textbf{Case 1:} The game $G+H$ is purely-atomic.
 Let $G=\left\langle\emptyset^{\ell}\,|\,\emptyset^{r}\right\rangle$ where $\ell\le r$.
 Then, by definition,
$\Conj{G} = \left\langle\emptyset^{-r}\,|\,\emptyset^{-\ell}\right\rangle$ and
 $\Conj{G}\in \GB$. Let $H=\pura{-\ell}{-r}$. Then $G+H\eqsim \pura{0}{0}=0$, but
 $H\in\GB$ if and only if $\ell=r$. Hence, the game $\Conj{G}$ is the inverse to
 $G$ if and only if $\ell = r$. Thus, a purely atomic game is invertible if and only if 
 it is a number. \\

In the below proof, because of numerous algebraic manipulations, we will revert to the short hand notation 
$-G = \Conj{G}$, if the existence of a negative is given by induction.\\

\noindent
\textbf{Case 2:} The game $G+H$ has at least one option.
We may assume that $G$ and $H$ are in their respective canonical forms.
Let $J = G+H \sim 0$.
It is a consequence of Theorem~\ref{thm:ettinger} that for all Left moves $J^L$,
 there exists $J^{LR}$ such that $J^{LR}\pr 0$.
 Without loss of generality, we will assume that $J^L=G^{L}+H$. There are two cases:\\

\noindent
\textbf{Case 2a:} Suppose there exists  a non-atomic-reversible $G^{L}\in\GL$.

(This option is not reversible, because only atomic-reversible options may exist in the reduced form.) We prove two claims.\\

\noindent \textit{Claim (i):} There exists $H^{R}$ such that $G^{L}+H^{R}\pr 0$.

  If there is a good Right reply  $G^{LR}+H\pr 0$, after adding $G$ to both sides
  (Theorem 14), we would have $G^{LR}\pr G$. This is a contradiction,
  since $G^L$ is not a reversible option. Therefore, there exists $H^R$ with
  $G^L+H^R\pr 0$.\\

\noindent \textit{Claim (ii):} With $H^R$ as in (i), $G^L+H^R\sim 0$.

  We have that $G^{L}+H^{R}\pr 0$. Suppose that $G^{L_1}+H^{R_1}\prec 0$, where 
  we index the options starting with $G^L=G^{L_1}$ and $H^R=H^{R_1}$.

 Consider the Right move in $G+H$ to $G+H^{R_1}$. Since $G+H\sim 0$, i.e., $G+H\su 0$,
 then there exists a Left option such that $(G+H^{R_1})^L\su 0$.

 Suppose that $G+H^{{R_1}L}\su 0$. Then, by adding $H$ to both sides,
 we get $H^{{R_1}L}\su H$. Therefore, since $H$ is in canonical form,
 $H^{R_1}$ is an atomic-reversible option and, by Theorem~\ref{thm:substitute}, 
 $H^{R_1} = r+\Num{n+1}$  for some real $r$ and nonnegative integer $n$.
 The two inequalities $G^{L_1}+H^{R_1}\prec 0$ and $G+H^{{R_1}L}\su 0$
 become $G^{L_1}+r+\Num{n}\prec 0$ and $G+r+\Num{n}\su 0$, respectively.
 In $G^{L_1}+r+\Num{n+1}$, against the Left move to $G^{L_1}+r+\Num{n}$,
 Right must have a move of the form $G^{{L_1}R}+r+\Num{n}\pr 0$. Since
 $r+\Num{n}$ is invertible, $G^{{L_1}R}+r+\Num{n}\pr 0$ and
 $G+r+\Num{n}\su 0$ leads to $G^{{L_1}R}\pr G$ (Lemma~\ref{lem:dis}),
  i.e., $G^{L_1}$ is reversible, which is a contradiction.

Consequently, we may assume that there is a non-reversible option, $G^{L_2}$,
 such that $G^{L_2}+H^{R_1}\su 0$. If $G^{L_2}+H^{R_1}\sim 0$ then,
 by induction, $H^{R_1}=-G^{L_2}$ (since $G$ and $H$ are in canonical form).
 Since $0\su G^{L_1}+H^{R_1} =  G^{L_1}-G^{L_2}$, then $G^{L_2}\su G^{L_1}$, which is a contradiction, because $\GL$ has no dominated options.
Therefore, $G^{L_2}+H^{R_1}\succ 0$. 

By Claim (i), there must exist a Right option in $H$, $H^{R_2}$, corresponding to $G^{L_2}$, and so on. If we repeat the process, we now have the following inequalities:

      \begin{eqnarray*}
      G^{L_1}+H^{R_1}\prec 0, &\quad& G^{L_2}+H^{R_1}\succ 0\\
      G^{L_2}+H^{R_2}\prec 0, &\quad&  G^{L_3}+H^{R_2}\succ 0\\
       \ldots ,
      \end{eqnarray*}
but the number of options is finite. Thus, without loss of generality, we may assume that
there is some $m$ such that $G^{L_{1}}+H^{R_m}\succ 0$  (re-indexing if necessary).

Because the inequalities are strict, summing the left-hand and the right-hand inequalities gives, respectively,
\[ \sum_{i=1}^m G^{L_i}+\sum_{i=1}^m H^{R_i}\prec 0\quad\mbox{ and } \quad
 \sum_{i=1}^m G^{L_i}+\sum_{i=1}^m H^{R_i}\succ 0\]
 which is a contradiction.

 Therefore, we conclude that $G^{L}+H^{R}\sim 0$ and, by induction,
 that $H^R= -G^L$.\\
 
\noindent
\textbf{Case 2b:} Suppose there exists an atomic-reversible option, $A\in\GL$.

Since $A$ is atomic-reversible, it follows, by Theorem~\ref{thm:substitute}, that
$A= \ell-(\Num{n+1})$, where $n$ is the minimum nonnegative integer such that $G\su \ell-\Num{n}$, and where $\ell = \Ls (B)$ is a real number (where $B$ is the reversing option). 
\begin{enumerate}
  \item[(i)] Suppose first that there is some Right option in $H$. We prove four claims.

  \begin{enumerate}[(a)]
   \item $\overline{Rs}(G)\geqslant \ell $.

Since $G\su \ell - \Num{n}$, we get $G+\Num{n}\su \ell $. Hence,
$\Rso(G)\ge Rs(G+\Num{n}) \ge \ell $, where the first inequality holds because Left can pass.

    \item There exists an atomic-reversible option $H^R\in\HR$.

Suppose not; we will argue that this implies $\overline{Rs}(G+H)>0$, contradicting
$G+H\sim 0$ (Theorem~\ref{thm:ettinger}). Because $H$ has no atomic-reversible
Right option, we saw in Case 2a that for all $H^R$ there exists non-atomic-reversible
$G^L$ such that $G^L+H^R\sim 0$. By induction, $G^L\sim -H^R$. Because
$A=\ell - \widehat{n+1}$ is an atomic-reversible option in $\GL$, by Theorem~\ref{thm:atomic}(2),
$\underline{Rs}(G^L)<\Rsu(A) = \ell$. Hence,
\begin{align}\label{wehaveseen}
\overline{Ls}(H^R)=-\underline{Rs}(-H^R)=-\underline{Rs}(G^L) > -\ell,
\end{align}
where the first equality is by definition of the conjugate of a game. This holds for all $H^R\in\HR$
 and so, $\overline{Rs}(H) > -\ell$.  Therefore, by Theorem~\ref{thm:pallowscores},
\begin{align}\label{asbefore}
\overline{Rs}(G+H)\geqslant \overline{Rs}(G)+\overline{Rs}(H)\geqslant \ell + \overline{Rs}(H) > \ell-\ell = 0,
\end{align}
and the claim is proved.
    \item The atomic-reversible Right option of $H$ is $-\ell + \Num{m+1}$ (where $m$ is minimum such that  $H\pr -\ell+\Num{m}$).

    We have seen in the inequality (\ref{wehaveseen}) that for all non-atomic-reversible $H^R$,
    $\Lso(H^R) > -\ell $. If the only atomic-reversible Right option of $H$ was
    $-s+\Num{m+1}$ and $\ell>s$, we would have $\overline{Rs}(H) > -\ell$,
    leading to the same contradiction as obtained in the inequality (\ref{asbefore}). Suppose, instead,
    that  the only atomic-reversible Right option of $H$ were $-s+\Num{m+1}$ with $\ell < s$.
    By definition of a reversing option (for an atomic-reversible Right option), we have that
    $H\pr -s+\widehat{m}$. Altogether,  $\underline{Ls}(H)\le Ls(H-\widehat{m})\leqslant -s<-\ell $.
    Therefore, by Theorem~\ref{thm:pallowscores}, $\Lsu(G+H)\le \Rsu(G)+\Lsu(H)\le \ell -s < \ell - \ell = 0$.
    The two contradictions together imply $s=\ell $.

         \item Finally, $m=n$.

Consider the integers, $n$ and $m$ as previously defined. They are minimal such that
$G\su \ell -\widehat {n}$ and $H\pr -\ell +\widehat{m}$, respectively. If $n\neq m$, say
$n<m$, from $G\su \ell -\widehat{n}$, adding $H$ to both sides gives
$0\su H+\ell -\widehat{n}\Rightarrow  H\pr  -\ell +\widehat{n}$. This is a contradiction
($m$ is not minimal). Hence, we must have $m=n$.
    \end{enumerate}
\noindent
Thus, we have proved that  if  $A = \ell -\widehat{n}$ (in reduced form) is a Left
atomic-reversible option of $G$, then there is an $H^R\in\HR$ with $H^R = -\ell +\widehat{n}=-A$.

  \item[(ii)] Since $A\in\GL$ is an atomic-reversible option, then $\HR$ is not an atom.

First, if it were true that  $\HR=\emptyset^{-s}$, for some real number $s$, then this would force $s=\ell$.
This follows by an  argument  similar to that in 2b(i.c).  $\overline{Rs}(G)\geqslant \ell$ holds by
2a (i.a). Thus, if $\overline{Rs}(H)=-s>-\ell$, then
$\overline{Rs}(G+H)\ge \Rso(G) + \Rso(H) > 0$.
Also, if $\overline{Rs}(H)=-s<-\ell $, then, because of the guaranteed property,
$\underline{Ls}(H)\leqslant -s<-\ell $. So, by $\underline{Ls}(G)=\ell $, we have
$\underline{Ls}(G+H)\le \Lsu(G) + \Lsu(H) < 0$.
The inequalities are contradictory, and so $s=\ell$.

Suppose therefore that $\HR=\emptyset^\ell$.
In this case, $A=\ell-\widehat{n+1}$ is the only Left option of $G$; any other options would be
 non-atomic-reversibles (by domination) paired in $H^\mathcal{R}$ (by Case 2a),
but there are none.
 Now, the non-atomic-reversible options of $\HL$ and $\GR$ are paired
 and since $G\in\GS$ then $\ell$ is less than or equal to all the  scores in the games of
 $G^\mathcal{R}$. Since $n\ge 0$ then,  by Theorem~\ref{thm:substitute},
 $\GL$ could be replaced by $\emptyset^{\ell}$ contradicting that $G$ was in reduced form.
  \end{enumerate}

We have seen that each $G^L$ has a corresponding $-G^L$ in the set of Right options of $H$.
This finishes the proof.
\end{proof}

As a final comment, not every game is invertible and we do not have a full 
characterization of invertible games. We do know that zugzwang games do not have inverses.

\begin{theorem} Let $G$ be a game with $Ls(G) < Rs(G)$. Then $G$ 
is not invertible.\end{theorem}

 \begin{proof}
 Suppose $G$ is a game with $\Ls(G) < \Rs(G)$. If $G$ is invertible then $G+\Conj{G}=0$, which by Theorem~\ref{thm:ettinger} implies that $\Lsu(G+\Conj{G})=0$.

Now,
\begin{eqnarray*}
\Lsu(G+\Conj{G}) &\leqslant& \Lsu(G) + \Lsu(\Conj{G})\qquad \mbox{(by Theorem~\ref{thm:pallowscores})}\\
&\leqslant& \Ls(G) + \Ls(\Conj{G}) \\
&=& \Ls(G) - \Rs(G)\\
&<& 0
\end{eqnarray*}
which contradicts $ \Lsu(G+\Conj{G})=0$ and finishes the proof.
\end{proof}

The converse is not true. For example, $G = \langle\langle -1\mid 1\rangle \mid 0 \rangle$ 
is not invertible, since $\Lsu(G+\Conj{G})=-1\ne 0$, and is not zugzwang since $\Ls(G)>\Rs(G)$.

\section{A Scoring Games Calculator}\label{sec:calc}

 The translation of a guaranteed game position to its canonical scoring value is not a trivial 
 computation task and cannot be done manually except for very simple examples.
 A computer program is required for more complex positions. 
 The Scoring Games Calculator (SGC) is such a program. It is implemented as a set of
 Haskell modules that run on an interpreter available in any Haskell distribution or
 embedded in a program that imports these modules.
 
 The SGC has two main modules, \texttt{Scoring} and \texttt{Position}, that act as containers of two 
 data types: \texttt{Game} and \texttt{Position}. The first module deals with scoring game values
 and the second with board positions given a ruleset.
 
 Game values represent values from set $\mathbb{S}$ like $\textless 1|\emptyset^{3} \textgreater$. This type
 includes an extensive list of Haskell functions that mirror the mathematical functions
 presented in this article. One simple example is predicate \texttt{guaranteed} that checks if a
 game value in $\mathbb{S}$ is also in $\mathbb{GS}$. Another operation is the sum of games that takes
 two values in $\mathbb{GS}$ and computes their disjunctive sum. 
% For example 
% $\textless 1|\emptyset^{3} \textgreater + \textless 1|\emptyset^{3} \textgreater$ returns 
% $\textless\textless 2|\emptyset^{4}\textgreater|\emptyset^{6}\textgreater$
 
 Position values represent board positions. Type \texttt{Position} is an abstract type. It encloses a set
 of services useful for all games, like reading a position from file or converting a
 position to its scoring value. These functions are only able to work when a concrete ruleset is implemented.
 Given a game, say Diskonnect, there should be a module \texttt{Diskonnect} 
 that imports module \texttt{Position}, and is required
 to implement the Diskonnect ruleset.
 Almost all effort to define a new game focus in the implementation of function \texttt{moves} 
 that, given a board position and the next player,
 returns the list of all possible next positions. With this, \texttt{Position} is able to construct a game tree
 for a given board position and to translate that position into its scoring value.
 
 The scoring universe together with its main theorems concerning reductions and comparisons 
 all have a strong recursive structure that fits quite well into a functional programming 
 language like Haskell. Not all mathematical definitions are simply translations to functions, 
 but some are. For example, the implementation of left-r-protected mirrors quite closely its definition,

\begin{verbatim}
lrp :: NumberData -> Game -> Bool
lrp r g = 
  ls_d g >= r &&
  for_all [ 
     for_any [lrp r gRL | gRL <- leftOp gR] | gR <- rightOp g]
\end{verbatim}

where \verb|ls_d| is 
$\underline{Ls}$ and syntax \texttt{[f x|x<-list]} defines list comprehensions.

 The SGC includes too many functions to be described here\footnote{The source code 
 and a user guide presenting all functionalities are available at 
 \url{https://github.com/jpneto/ScoringGames}}. Currently, the following 
guaranteed rulesets are implemented: Diskonnect, Kobber, TakeSmall and TakeTall.

%%%%%%%%%%%%%%%%%%%%%%%%%%%%%%%%%%%%%%%
%%%%%%%%%%%%%%%%%%%%%%%%%%%%%%%%%%%%%%%
%%%%%%%%%%%%%%%%%%%%%%%%%%%%%%%%%%%%%%%

\bibliographystyle{plain}
\bibliography{games4}

\end{document}